\documentclass[12pt]{article}

\usepackage{amsmath}
\usepackage{amssymb}
\usepackage{xspace}
\usepackage{graphicx, subfig}               
\usepackage{setspace}
\usepackage[authoryear, longnamesfirst, sort&compress]{natbib}
\shortcites{AlnaesBlechta2015a}

\newtheorem{theorem}{Theorem}[section]
\newtheorem{proof}{Proof}[section]
\newtheorem{lemma}[theorem]{Lemma}

\newtheorem{corollary}[theorem]{Corollary}
\newtheorem{assumption}[theorem]{Assumption}

\usepackage{color}
\newcommand{\mc}{\color{red} FER}  
\newcommand{\cm}{ \color{black}} 
\newcommand{\cutmc}[1]{\mc [cut] \cm} 

\newcommand{\ac}{\color{red} }  
\newcommand{\ca}{ \color{black}} 
\newcommand{\cutac}[1]{\ac [cut] \ca} 

\newcommand{\N}{\mathbb{N}}
\newcommand{\R}{\mathbb{R}}
\newcommand{\Z}{\mathbb{Z}}
\newcommand{\Q}{\mathbb{Q}}

\newcommand{\Y}{\mathcal{Y}}
\renewcommand{\H}{\mathcal{H}}
\newcommand{\F}{\mathcal{F}}

\title{Posterior distribution existence and error control in Banach spaces
in the Bayesian approach to UQ in inverse probelms}

\author{J. Andr\'es Christen\footnotemark[2]\ \footnotemark[3] \and
Marcos A. Capistr\'an\footnotemark[2] \and
María Luisa Daza-Torres\footnotemark[2] \and
Hugo Flores-Arg\"uedas\footnotemark[2] \and
J. Cricelio Montesinos-L\'opez\footnotemark[2]
}

\footnotetext[2]{Centro de Investigaci\'on en Matem\'aticas (CIMAT), Jalisco S/N, Valenciana, Guanajuato, GTO, 36023, MEXICO.
\textit{jac, marcos, mdazatorres, jose.montesinos, moreles,  at cimat.mx}}
\footnotetext[3]{Corresponding author.}

\date{10OCT2018}

\begin{document}

\maketitle

\begin{abstract}
We generalize the results of \cite{Capistran2016} on expected Bayes factors (BF) to control the numerical error in the posterior distribution to an infinite dimensional setting when considering Banach functional spaces and now in a prior setting. 
The main result is a bound on
the absolute global error to be tolerated by the Forward Map numerical solver, to keep the BF
of the numerical vs. the theoretical model near to 1, now in this more general setting, possibly including
a truncated, finite dimensional approximate prior measure.  In so doing we found a far more general setting
to define and prove existence of the infinite dimensional posterior distribution than that depicted in, for example, \cite{Stuart2010}.  Discretization  consistency and rates of convergence are also investigated in this general setting for the Bayesian inverse problem.  
\end{abstract}


\newpage

\section{Introduction}\label{sec:intro}

Bayesian UQ in a nutshell is Bayesian inference on a (possibly infinite dimensional) parameter
$\theta$ with data $y_i$ such that, for example,
$$
y_i = \F(\theta) + \epsilon_i; ~\epsilon_i \sim N( 0, \sigma^2) .
$$
The regressor $\F(\theta)$, or Forward Map (FM), is commonly a complex non-linear map arising from unknown parameters
in a system of ODEs or PDEs.  Then to evaluate $\F(\theta)$ we require to solve a system of (O,P)DEs.
Not only that, but this commonly involves a numerical solution with some error $\F^{\alpha(n)}(\theta)$, which is the actual
regressor we can work with in our computer.  A prior $\pi(\cdot)$ is stated for $\theta$ and a \textit{numerical} posterior
distribution is obtained.  $\alpha(n)$ represents a discretization used to approximate the FM and as $n$
increases the discretization becomes finer and the approximation becomes tighter.

In this paper we are concerned with the numerical error induced in this posterior
in comparison to the theoretical posterior (when considering the exact theoretical FM $\F(\theta)$) and also on the error introduced in the numerical posterior when using a \textit{truncated}, finite dimensional prior $\pi_k$.
Moreover, we dicuss practical guidelines to choose the numerical discretization
refinement $n$ and the priori truncation $k$ in order to have correct posterior numerical error control.
We consider a general, not necessarily Gaussian, model for the data $y_i$s. 

\cite{Capistran2016} discuss the latter and this paper generalizes their results to
functional spaces, including a discretization/truncation
of the prior.  \cite{Capistran2016} use a posterior bound (once the data is seen) and requires the
estimation of normalizations constants.  Here we use a prior(predictive) bound that results
in a global bound for the FM to control the numerical error; a brief review of \cite{Capistran2016} and
its shortcomings is given below in section~\ref{sec:intro_EABF}. 

Undoubtedly, the first step is to define the posterior distribution in a general setting including
infinite dimensional spaces.  In the context of Bayesian inverse problems \cite{Stuart2010} did several
advances and found regularity conditions for the posterior to exists in a fairly general setting
\citep[see][and references therein]{Stuart2010}.
The normalization constant in the posterior is proven to be finite and positive and thus the posterior is indeed
a probability measure using boundedness assumptions on the likelihood and considering
Gaussian priors \citep[][assumption 2.7(i,ii) and theorem 4.1]{Stuart2010}. 
Recently \cite{Hosseini2017} generalized the latter now considering priors with exponentially decaying tails,
using the same regularity conditions.

However, to our surprise, in studying the mentioned results we found out that in other contexts defining
the posterior distribution in general spaces is a very well known task; a nice example is contained in the text book
\cite{Schervish1997}.  A very powerful tool that can be used here is \textit{disintengration}, although it is not
essential.  The principal remark here is that the existence of the posterior distribution can be established in
a far more general sense than what  \cite{Stuart2010} establishes and these results are well known in the
general Bayesian literarure.  Below we discuss the existence of the
posterior distribution in this perspective.

\subsection{Existence of the posterior distribution in infinite dimensional spa\-ces} 

It always puzzle us that in any other context of Bayesain inference we need not worry for, for example, the prior tail
behaviour \citep{Stuart2010} or, in fact, \textit{any other condition for the posterior to exists}.
The usual practice is to define a parametric model for data $y$, $f(y | \theta)$, a prior for the parameter $\pi(\theta)$ and without guilt
and further protection we declare $f(y | \theta)\pi(\theta)$ to be a joint distribution on $( y, \theta )$; the
usual argument being that, it is indeed positive and $\int \int f(y | \theta)\pi(\theta) dy d\theta = 1$.
But, when does $f(y | \theta)\pi(\theta)$ define a joint distribution? when, to start with, the latter
integrals exist and swap?  But in any case, we depart from the construction of a joint probability
measure for both $(y,\theta)$.

What we call modern Bayesian statistics, in its foundations, requires exactly that: a joint probability measure
$P$ on the whole measurable space  $(\Omega, @)$ of uncertain events, both observable, $y$, or not, $\theta$.
The existence of such measure $P$ is proven by assuming a set of axioms on a preference relationship
on events on $@$ based on a system of bets performed by an \textit{agent}.
Conditional on the chosen space $(\Omega, @)$ and on the agent preferred system of bets
$P$ quantifies the agent's ``uncertainty'' on $@$ (namely a system of bets comprising the axioms),
and this is the basis for the epistemic or conditional probabilistic
or Bayesian \citep[or which some also like to call,
lightly or pejorative, ``subjective'';][]{christen2006}  approach to Uncertainty Quantification.
Our preferred axiomatic development is that of \cite{DeGroot1970}.  

In the same axiomatic development, if then an event $D \in @$ is observed, a new system of bets is
precluded in which bets on events are only relevant in terms of the intersection of those events with $D$,
ie. anything outside $D$ ceases to be relevant.  The existence of a new measure $P_D$ on $(\Omega, @)$
is guaranteed, which coincides with the new updated system of bets after  $D \in @$ has been observed
and it turns out that
$$
P_D( A ) = P( A | D) = \frac{P( D | A) P(A)}{P(D)}
$$
for all $A \in @$.  That is, given the set of axioms, the updated measure $P_D$ is precisely the
conditional probability conditional on $D$.  All inferences, given that we observed $D$, stem from
the conditional probability $P( A | D)$, namely the posterior or \textit{a posteriori} probability measure.
The way we perform any necessary calculations to obtain $P( A | D)$, exactly or approximately,
is up to us, and certainly Bayes theorem is used in most cases (not always, eg. when calculating a
predictive posterior only total probability is used).  Note therefore that
Bayes theorem is not the fundamental issue in modern Bayesian statistics, nor its interpretations
give meaning to modern Bayesian UQ.

However, a problem arises when modeling data with continuos distributions, since
realized data $D = \{Y = y\}$ have $P(D) = 0$ and the above simple calculation of $P_D(A) = P( A | D)$ cannot
be used.  Fortunately, this is a classical problem in probability, since conditioning on events of
zero probability is a necessity well beyond Bayesaian statistics.  Kolmogorov studied the
problem but the modern approach, for very many technical reasons, is called 
\textit{disintegration}.  A very nice review may be found in \cite{Chang1997}, specifically example 9
discusses the definition and existence of the posterior distribution. \cite{Leao2004} also present a nice review.

Disintegration has the correct properties as a conditional distribution, now generalized to
events of probability zero.  In particular, $\Omega - D$ becomes irrelevant.  The bottom line is the
same as in \cite{Stuart2010}: the posterior measure has as density the likelihood function w.r.t
the prior measure.  However, the posterior may be proven to exists in a very general setting without any regard to tail
behaviour of the prior etc.  For completeness, all these results are presented in detail in section \ref{sec:setting}.

As it turns out, a good enough regularity setting is this:
$f(y | \theta)$ is continuos in $\theta$ and the joint measure space $(\Omega, @)$ is Polish,
leading to a Radon joint measure $P$, see lemmas \ref{lem:joint} and \ref{lem:bayes}.
As far as section~\ref{sec:setting} is concerned, we stress the fact that
only the former we consider a relevant observation on our part (continuity of the likelihood),
the rest in that section is based on classical probability results and are well known in
other areas of Bayesian statistics.

\subsection{Consistency, convergence and EABF}\label{sec:intro_EABF}

As mentioned above, we are interested in establishing guidelines for choosing a discretization
level $\alpha(n)$ for the FM and a truncation for the prior $\pi_k$.  The problem is addressed in
\cite{Capistran2016} in the finite dimensional case and here we generalize their results
for parameters in infinite dimensional Banach spaces and a truncation in the prior distribution.

\cite{Capistran2016} present an approach to address the above problem
using Bayes factors (BF; the odds in favor) of the numerical model vs the theoretical
model (further details will be given in section  \ref{sec:setting}). 
In an ODE framework, these odds are proved in \cite{Capistran2016} to converge to 1,
that is, both models would be equal, in the same order as the numerical solver used.
For high order solvers \cite{Capistran2016} illustrates, by reducing the step size
in the numerical solver,  that there should exist a point at which the BF is basically 1,
but for fixed  discretization $\alpha(n)$ (step size) greater than zero.  This is the main point made by \cite{Capistran2016}:
it could be possible to calculate a threshold for the tolerance such
that the numerical posterior is basically equal to the theoretical posterior so, although we
are using an approximate FM, the resulting posterior is nearly error free.
\cite{Capistran2016} illustrate, with some examples, that
such optimal solver  discretization leads to basically no differences in the numerical and
the theoretical posterior, since the BF is basically 1; potential saving CPU time by choosing a corser solver.

However,  \cite{Capistran2016} still has a number of shortcomings.  First, it depends crucially
on estimating the normalizing constants from Monte Carlo
samples of the unnormalized posterior, for a range of  discretizations $\alpha(n)$.  This is
a very complex estimation problem, subject of current research, and is in fact
very difficult to reliably estimate these normalizing constants in mid to high dimension
problems.  Second, \cite{Capistran2016} approach is as yet incomplete
since one would need to decrease
$\alpha(n)$ systematically, calculating  the normalization constant of the corresponding numerical posterior
to eventually estimate the normalization constant of the theoretical posterior
\citep[see figure 2 of][]{Capistran2016}, which in turn will pin point a  discretization at which both models are
indistinguishable.  Being this a second complex estimation problem, the main difficulty
here is that one has already calculated the posterior for small step sizes and therefore
it renders useless the selection of the optimal step size.

To improve on \cite{Capistran2016}, the idea of this paper
is to consider the \textit{expected} value of the BFs, before data is observed.  We will try to bound this
expected BF to find general guidelines to establish error bounds on the numerical solver,
depending on the specific problem at hand and the sample design used, but not on particular
data.  These guidelines will be solely regarding the forward map and, although perhaps conservative,
represent useful bounds to be used in practice.  Moreover, as already mention, we generalize
\cite{Capistran2016} to an infinite dimensional setting and also considering a truncation
in the prior. 

The basic idea then is to establish the relative merit of the numeric model vs. the theoretical model using
Bayesian model selection.

We first prove that the approximations are consistent.  That is, that the numerical
posterior converges to the theoretical posterior.  This has been proved, and discussed extensively,
using the Hellinger distance \citep[eg.][]{Stuart2010}.  Also, rates of convergence have been discussed elsewhere \citep{Stuart2010, Bui-Thanh2014}.
Here in section~\ref{sec:weak}, in the more general setting considered in this paper and for completeness, we use weak convergence.
Then to establish the consistency in the rate of convergence in section~\ref{sec:TV} we use the Total Variation norm.

Having this we prove our main result, for Banach spaces, for the Expected Absolute difference of the BF to 1 (EABF),
considering any location-scale family
for the distribution of the data; the main results of the paper are found in section~\ref{sec:EABF}.
In section~\ref{sec:base} we consider the prior truncation and  in section~\ref{sec:examples} a series of examples.

For the moment we finish this introduction with a brief discussion on the use of weak convergence
and the Total Variation (TV) norm.

\subsection{Weak convergence and the Total Variation norm}

In probability theory, the basic convergence criterion is weak convergence.
Other convergence criteria (in probability, in TV, in Lp etc.) are commonly generalized from weak convergence
 \citep{Billingsley1968}.
Probability measures $\mu_k$ weakly converge to  $\mu$ if the Lebesgue integrals
$\int f(x) \mu_k(dx)$ converge to $\int f(x) { \mu(dx) }$ for all measurable, non-negative,
continuous bounded functions $f$.  We write $\mu_k \Rightarrow \mu$.

In oder to have a clear concept of rates of convergence we require a metric to measure distance
between the involved objects.  Total variation (TV) is one of the most common for many reasons
\citep{Gibbs2002}.  The TV distance between two measures $\mu_1$ and $\mu_2$
on the same measure space $(\Omega, @)$ is defined as
$$
|| \mu_1 - \mu_2 ||_{TV} = \sup_{A \in @} |\mu_1(A) - \mu_2(A)| =
\frac{1}{2} \max_{|h| \leq 1} \left| \int h(x) \mu_1(dx) - \int h(x) \mu_2(dx) \right| ,
$$
where $h : \Omega \rightarrow \R$ measurable.  Note that, if $\mu_1$ approximates the posterior distribution
$\mu$ then $|| \mu_1 - \mu ||_{TV}$ is the upper bound for the difference in any posterior probability we wish
to calculate and/or on the error in any bounded posterior expectation we need to calculate.  Moreover, note that utility functions are
bounded and with correct units belong to $[0,1]$ \citep{DeGroot1970}.
Then $|| \mu_1 - \mu ||_{TV}$ is the maximum error incurred in calculating expected utilities when using
$\mu_1$ instead of $\mu$.  As far as Bayesian theory is concerned, TV is quite well suited for what is required.

Indeed, the Hellinger distance could be used as well, as has been the tradition in the Bayesian UQ context.
Note however
that TV is equivalent to Hellinger \citep{Gibbs2002}; convergence in TV implies convergence
in Hellinger and viceversa.  It bounds perhaps to facility in proofs and direct interpretation and that is why
we choose TV.


\section{Setting and existence lemmas}\label{sec:setting}

Let $Y \in \Y \subset \R^m$ be the data at hand and $\{ P_\theta : \theta \in \Theta \}$ be a family of probability
models for $Y$.  We assume that the family of probability models for the observables $Y$
have a density $f_\theta (y)$ w.r.t a $\sigma$-finite measure $\lambda$, namely a product of the Lebesgue and counting
measures in $\R^n$ to accommodate, possibly, discrete and continuous observations.  That is
$$
P_\theta ( Y \in A ) = \int_A P_\theta (dy) = \int_A f_\theta (y) \lambda(dy) ,
$$
for all measurable $A$.  For example, $\Y$ is a product space of subsets of $\R$ or $\Z$, leading
to discrete and/or continuos data.    This is the usual setting in parametric inference.

In any case, with the usual topological considerations we assume $\Y$
is a \textit{Polish} space.  Polish spaces include complete metric spaces that have a countable dense subset.
$\Y$ should be viewed as a Polish space with the standard metric in $\R$ and the discrete metric in
$\Z$, and $\lambda$ then results in a Borel $\sigma$-finite measure on $\Y$.
$P_\theta$ is then a Radon measure for all $\theta \in \Theta$, since any Borel probability measure
on a Polish space is Radon.  We use this last fact in the
proof of lemma \ref{lem:bayes} below.

Until now the parameter space $\Theta$ is arbitrary.  We need to define a measurable space $(\Theta, @)$ to be able to
define a probability measure $\pi$ on $\Theta$, namely a prior distribution.  So far $f_\theta (y)$ cannot be considered
a conditional distribution but due to the next two lemmas we adopt the more common notation $f( y | \theta) = f_\theta (y)$.

\begin{lemma}\label{lem:joint}
Let $g: \R^m \times \Theta \rightarrow \R^+$ be any $\lambda \times \pi$-measurable function.  If
$g(y, \theta)  f( y | \theta)$ is a $\lambda \times \pi$-measurable function then
$$
\int g(y, \theta) \Q( dy, d\theta ) := \int \int g(y, \theta) P_\theta (dy) \pi(d\theta ) = \int \int g(y, \theta)  f( y | \theta)  \lambda(dy) \pi(d\theta)
$$
defines a joint probability measure $\Q$ on the product space ${\R^m} \times \Theta$.
\end{lemma} 

\begin{proof}
Since $\lambda$ and $\pi$ are $\sigma$-finite ($\pi$ is finite) then
by Tonelli's theorem $\theta \mapsto \int  g(y, \theta)  f( y | \theta) \lambda(dy)$ is measurable,
$g(y, \theta)  f( y | \theta)$ is (non-negative)
$\lambda \times \pi$-integrable and the above integrals
swap.  Moreover, using $g \equiv 1$ we have
$\Q( \Y \times \Theta ) = \int \int f( y | \theta) \lambda(dy) \allowbreak \pi(d\theta) = 1$.  See for example \cite{Schervish1997}, p. 16.
\end{proof}

\begin{lemma} [Bayes' theorem]\label{lem:bayes}
If $\Theta$ is a separable Banach space and $\theta \mapsto f( y | \theta)$ is continuos for all
$y \in \Y$ then:
\begin{enumerate}
\item The joint measure $\Q$ exists, as defined in Lemma \ref{lem:joint}.
\item The $\theta$-disintegration $\Q_\theta$ of  $\Q$ exists, $P_\theta$ may be seen as such $\theta$-desintegration and therefore $f(y | \theta)$ may be seen as the conditional density of $Y$ given $\theta$.
\item The $y$-disintegration $\Q_y$ of $\Q$ exists and is the general definition of the conditional measure
$\Q( \cdot | Y = y)$ on $\Theta$ given $Y = y$, namely, the posterior distribution.
\item Moreover, for any measurable $g$ we have
$$
\int g(\theta)  \Q_y( d\theta ) = \frac{\int g(\theta) f(y | \theta) \pi(d\theta)}{\int f(y | \theta) \pi(d\theta)} ,
$$
that is $\frac{\partial \Q_y}{\partial \pi} \propto f(y | \theta)$, for all $y \in \Y$.
\end{enumerate}
\end{lemma} 

\begin{proof}
The $\lambda \times \pi$-measurability of $f( y | \theta)$ was proven in \cite{Gowrisankaran1972}.
Since $g(y, \theta)  f( y | \theta)$ is also $\lambda \times \pi$-measurable, from Lemma~\ref{lem:joint},
1. above follows.  Moreover, any separable Banach space is a Polish space and the product
space $\Y \times \Theta$ is also Polish, therefore the joint probability measure $\Q$ is a Radon measure
and the prior $\pi$ is also Radon.
The rest follows from standard results in disintegration with Radon probability measures,
see example 9 of \cite{Chang1997}.  This is also proven in, for example, \cite{Schervish1997}, p. 16, although not
using the disintegration argument.
\end{proof}

\subsection{Remarks on Lemmas \ref{lem:joint} and \ref{lem:bayes}} 
\begin{itemize}
\item \textbf{Generality:} The combination of lemas \ref{lem:joint} and \ref{lem:bayes}
state the existence of the posterior measure, which are based on standard results in probability and integration.  Note
that we do not require any restriction on the tail behavior on the likelihood nor on the prior.   This is a far more general result
than \cite{Stuart2010} or \cite{Hosseini2017}.  Existence of the posterior measure in the parametric setting
is guaranteed with the continuity of the likelihood and regularity of the underlying space,
namely a Polish space.  

\item \textbf{Continuous likelihood:} Note that for each $\theta, f( y | \theta)$ is a $\lambda$-measurable function.
With continuity on $\theta$ it follows that $f( y | \theta)$ is $\lambda \times \pi$-measurable.
This is indeed a profound result in measure theory
that puzzled topologists for many years \citep[eg.][]{Sierpinski1920}.
The reference we use \citep{Gowrisankaran1972} made his prove for when $\Theta$
is a Suslin space, which is a generalization
of Polish spaces.  Counterexamples showing that a measurable function on each variable separately is not
measurable in the product space show that the continuity requirement on $\theta \mapsto f( y | \theta)$ may not
be relaxed without further provisions.  That is, the likelihood is required to be continuous.  

\item \textbf{Cromwell's rule:} If an event has zero a priori probability then it will have zero posterior probability; indeed since $\Q_y[\pi] << \pi$.  We will adopt the notation for the posterior measure
$\Q_y[\pi]$ to make the dependance explicit both on the data $y$ and on the prior $\pi$.
In this respect $\Q_y[\pi]$ may be seen as an operator that transforms (updates) the prior measure $\pi$
into the posterior measure $\Q_y[\pi]$, which represents the inference process of learning from the data $y$.

\item \textbf{Likelihood principle:} As usual, Bayesian inference follows the likelihood principle since the
posterior measure depends on the data only through the likelihood.  ``Well-Posedness'' as studied by
\cite{Stuart2010} or \cite{Hosseini2017}, in which close enough data $y$ and $y'$ will lead to similar posteriors, is
interesting but we believe is a wrong concept.  Two very different data sets should lead to the same inferences (eg. $y$
and $y'$ having the same mean) and even two alternative models should lead to the same conclusions, when following
the likelihood principle (eg. binomial vs. negative binomial sampling); see for example \cite{Berger1988}.  

\item \textbf{Prior predictive measure:} As usual, from lemmas \ref{lem:joint} and \ref{lem:bayes} we see that the normalization constant,
or partition function, for the posterior
$$
Z(y) = \int f(y | \theta) \pi(d\theta)
$$
now viewed as a function of $y$ is in fact the marginal density, w.r.t $\lambda$, of the joint measure $\Q$.  That is,
is a density for not yet observed data $Y$, namely the prior predictive measure.  Defining the posterior through
Radon-Nikodym derivatives does not preclude directly the existence of such measure.  
\end{itemize}

In the next section we discuss how to ensure that when substituting the likelihood with a numeric approximation
$f^{n}( y | \theta)$, the corresponding posterior $\Q_y^n[\pi]$ is close enough to the theoretical
posterior $\Q_y[\pi]$.  Also we will discuss the analogous when using an alternative prior $\pi_k$ instead of
$\pi$ and combining both, leading to the approximate posteriors $\Q_y[\pi_k]$ and $\Q_y^n[\pi_k]$.

\section{The inverse problems setting and discretization consistency}\label{sec:weak}

We follow the general setting of \cite{Scheichl2017} for the statistical inverse problem.
Let $\Theta$ and $V$ be separable Banach spaces, let $\F: \Theta \rightarrow V$ be the Borel measurable
forward map (FM) and $\H: V \rightarrow A \subseteq \R^{m+s}$ the Borel measurable observation operator.
The composition $\H \circ \F$ defines a Borel measurable mapping from the parameter space $\Theta$ to the data sample space
in $R^m$, plus possibly additional parameters.
Going beyond Gaussian noise assume that $f_o( y | \eta )$ is a density for data $y$ w.r.t. $\lambda$
for all $\eta \in A$.  The parametric family of sample models, as in section~\ref{sec:setting}, is defined with
the family of $\lambda$-densities
$$
f( y | \theta) =  f_o( y | \H(\F(\theta)) ); \theta \in \Theta .
$$

To fix ideas we elaborate the usual independent Gaussian noise case, 
$$
f_o( y | \eta) = \prod_{j=1}^m \sigma^{-1} \rho\left( \frac{y_j - \eta_j}{\sigma} \right)
$$
and $\rho(x) = \frac{1}{\sqrt{2\pi}} e^{-\frac{x^2}{2}}$, ie. $y_i = \H_j(\F(\theta)) + \sigma \epsilon_j; \epsilon_j \sim N(0,1)$.
If $\sigma$ is also unknown we may take $s=1$ and include it as a parameter.  The same if we had and unknown variance-covariance matrix etc.  We do not discuss this case further in the main part of the paper.  Some notes are added in section \ref{sec:diss}
regarding the case when $\sigma$ is unknown.

Let $\F^{\alpha(n)}$ be a discretized version of the forward map $\F$, for some discretization $\alpha$ that depends
on an integer refinement $n$.  For example, a time step size, FEM discretization, etc.  This is the actual numerical
version of the forward map defined in our computers.  Let
$f^{n}( y | \theta) =  f_o( y | \H(\F^{\alpha(n)}(\theta)) )$ be the resulting discretized numerical
likelihood.   Moreover, suppose there are approximate or alternative prior measures $\pi_k$
also defined in $\Theta$.  In the rest of the paper we take the following assumption.

\begin{assumption}\label{assp:1}
\textit{Assume that, for all $y \in \Y$  the observation model
$f_o( y | \eta)$ is uniformly Lipschitz continuous for each $\eta$, and for $y \in \Y ~\lambda$-a.s. $f_o( y | \eta)$ is bounded. 
Moreover, the FM maps $\H \circ \F$ and $\H \circ \F^{\alpha(n)}$ are continuous.}
\end{assumption}

If $\H \circ \F$ and $\H \circ \F^{\alpha(n)}$ are continuous then
$\theta \mapsto f(y | \theta)$ and $\theta \mapsto f^{n}(y | \theta)$ are continuous and all requirements are met for
lemmas \ref{lem:joint} and \ref{lem:bayes} and the posterior measures are well defined and exist as
probability measures when using the theoretical likelihood and exact prior $\Q_y[\pi]$ and also
when using the numerical likelihood or/and an alternative prior, namely
$\Q_y^n[\pi]$, $\Q_y[\pi_k]$ and $\Q_y^n[\pi_k]$.  Also let
$Z^n(y)$, $Z_k(y)$ and $Z^n_k(y)$ be the corresponding partition functions in each case.
In the usual setting of \cite{Stuart2010},  \cite{Scheichl2017} and others
it is also assumed that $\H(\F(\theta))$ is continuous; here we require nothing further. 

Note that if we consider independent data with a location-scale model as
\begin{equation}\label{eqn:loc_scale_model}
f_o( y | \eta) = \prod_{j=1}^m \sigma^{-1} \rho\left( \frac{y_j - \eta_j}{\sigma} \right)
\end{equation}
where $\rho(x)$ is uniformly Lipschitz continuous and $\sigma$ known, the first part of assumption \ref{assp:1} is met
and we only require to establish that $\H \circ \F$ and $\H \circ \F^{\alpha(n)}$ are continuous.
Indeed the former is true if $\rho(x)$ is Gaussian.

Assume a global error control of this numeric FM as
\begin{equation}\label{eqn:global_error1}
||  \H(\F(\theta)) - \H(\F^{\alpha(n)}(\theta)) || < K_0 |\alpha(n)|^p ,
\end{equation}
for some functional $| \cdot |$.  Note that this is a global bound, valid for all $\theta \in \Theta$
and includes already the observational operator.  That is, it is a global bound (for all $\theta \in \Theta$)
but is only a statement at the locations $\H_j$s where each $y_j$ is observed.

Usually the error control global bounds are proven for the
FM but these are easily inherited to the composition $\H \circ \F$ by ensuring, for example, that
$\H$ is  Lipschitz continuous as we next explain.  From assumption \ref{assp:1}
$f_o( y | \eta)$ is uniform Lipschitz continuous for any given $y$.
Then since $| f_o( y | \eta) - f_o( y | \eta') | < L | \eta - \eta' |$ we have
\begin{equation}\label{eqn:global_error2}
| f^{n}( y | \theta) - f(y | \theta) | = | f_o( y | \H(\F^{\alpha(n)}(\theta)) ) - f_o( y | \H(\F(\theta)) ) | <  K_1 |\alpha(n)|^p ,
\end{equation}
which is also a global error bound, now for the numeric likelihood, where the constant $K_1 = L K_0$ is independent of
$\theta$. 

The next step is to prove the consistency of using the discretization and the prior \textit{truncation} $\pi_k$
(the term will be clear in section \ref{sec:base}), that is, how $\Q_y^n[\pi]$ and $\Q_y[\pi_k]$ tend
to the theoretical posterior measure $\Q_y[\pi]$.  We first prove the latter in weak convergence.
Rates of convergence are proven in the then Total Variation norm in the following section.
As mentioned before, we stress the fact that similar consistency results have proved before
in this Bayesian inverse setting, in a more particular setting.
We present weak convergence and TV rates of convergence results
since our setting is more general basically only requiring assumption \ref{assp:1}.

\subsection{Weak convergence}

The following theorem presents our discretization consistency results.

\begin{theorem}[discretization consistency]\label{teo:weak1}
With assumption \ref{assp:1}:
\begin{enumerate}
\item With the FM approximation result in (\ref{eqn:global_error2}), then
$\Q_y^n[\pi_k] \Rightarrow \Q_y[\pi_k]$ and $\Q_y^n[\pi] \Rightarrow \Q_y[\pi]$ (as $n \rightarrow \infty$).

\item If $\pi_k \Rightarrow \pi$
then $\Q_y^n[\pi_k] \Rightarrow \Q_y^n[\pi]$ and $\Q_y[\pi_k] \Rightarrow \Q_y[\pi]$ (as $k \rightarrow \infty$).
\end{enumerate}

\end{theorem}

\begin{proof}
1. From (\ref{eqn:global_error2})
we have that $f^n( y | \theta) \rightarrow f(y | \theta)$ for all $\theta \in \Theta$, then by bounded convergence
$$
Z^n_k(y) = \int f^n( y | \theta) \pi_k(d\theta) \rightarrow \int f( y | \theta) \pi_k(d\theta) = Z^{k}(y)
$$
since $\pi_k$ is finite \citep[][chap. 3]{Swartz1994}.  Since $[Z^n_k(y)]^{-1}  \rightarrow [Z^{k}(y)]^{-1} > 0$
we also have $\frac{f^n( y | \theta)}{Z^n_k(y)} \rightarrow \frac{f(y | \theta)}{Z^{k}(y)}$ for all $\theta \in \Theta$.
Now, since $\Q_y^n[\pi_k]$ and $\Q_y[\pi_k]$ have
the latter as densities w.r.t $\pi_k$ this implies $\Q_y^n[\pi_k] \Rightarrow \Q_y[\pi_k]$ by Scheff\'e's lemma.
The prove for $\Q_y^n[\pi] \Rightarrow \Q_y[\pi]$ is analogous.

2. Note that $f^n(y | \theta)$ is bounded,
real, non-negative, continuos function, therefore
$$
Z^n_k(y) = \int f^n( y | \theta) \pi_k(d\theta) \rightarrow \int f^n( y | \theta) \pi(d\theta) = Z^{n}(y) .
$$
Let $g(\theta)$ be any bounded, real, non-negative, continuos function, then since
$[Z^n_k(y)]^{-1} \allowbreak \rightarrow  [Z^{n}(y)]^{-1}$ and
$\int g(\theta) f^n( y | \theta) \pi_k(d\theta) \rightarrow \int g(\theta) f^n( y | \theta) \pi(d\theta)$ then
$$
[Z^n_k(y)]^{-1}\int g(\theta) f^n( y | \theta) \pi_k(d\theta) \rightarrow [Z^{n}(y)]^{-1}\int g(\theta) f^n( y | \theta) \pi(d\theta)
$$
which implies $\Q_y^n[\pi_k] \Rightarrow \Q^n_y[\pi]$.  The prove for $\Q_y[\pi_k] \Rightarrow \Q_y[\pi]$ is analogous.

\end{proof}

\subsection{Total variation and rates of convergence}\label{sec:TV}

As previously mentioned we use TV to establish rates of convergence in our discretizations.

\begin{theorem} \label{teo:tvn_rate}
Assume \ref{assp:1} and the rate of convergence in (\ref{eqn:global_error2}) then
\begin{eqnarray*}
|| \Q_y^n[\pi_k]  - \Q_y[\pi_k] ||_{TV} & < & \frac{K_1}{Z_k(y)} |\alpha(n)|^p  \\
& \text{and} & \\
|| \Q_y^n[\pi] - \Q_y[\pi]||_{TV} & < & \frac{K_1}{Z(y)} |\alpha(n)|^p  \\
\end{eqnarray*}
for big enough $n$.
\end{theorem}

\begin{proof}
This is proven in lemma \ref{lemm:conv_post2}.
\end{proof}

\begin{theorem}\label{teo:tvk_rate}
With assumption \ref{assp:1}, if $|| \pi_k - \pi ||_{TV} \rightarrow 0$ then
\begin{eqnarray*}
|| \Q_y^n[\pi_k]  - \Q_y^n[\pi] ||_{TV}  & < & \frac{f^{n}(y | \hat{\theta}_n)}{Z^n(y)} || \pi_k - \pi ||_{TV}\\
& ~~\text{and}~~ & \\
|| \Q_y[\pi_k] - \Q_y[\pi]||_{TV} & < & \frac{f (y | \hat{\theta})}{Z(y)}  || \pi_k - \pi ||_{TV} \\
\end{eqnarray*}
for big enough $k$, where $\hat{\theta}_n, \hat{\theta} \in \Theta$ maximize $f^{n}(y | \cdot )$ and $f( y | \cdot )$. 
\end{theorem}

\begin{proof}
For $h$ measurable with $|h| \leq 1$ we have
\begin{eqnarray*}
\left| \int h(\theta) f^n( y | \theta) \pi_k(d\theta) - \int h(\theta) f^n ( y | \theta) \pi(d\theta) \right|  & = &
\left| \int h(\theta) f^n( y | \theta) (\pi_k - \pi)(d\theta) \right| \\
& \leq & \int | h(\theta) | f^n (y | \theta) |\pi_k -\pi |(d\theta) .
\end{eqnarray*}
Let $b_k = \int h(\theta) f^n( y | \theta) \pi_k(d\theta)$ and $b = \int h(\theta) f^n( y | \theta) \pi(d\theta)$, the above implies
$|b_k - b| \leq f^n( y | \hat{\theta}_n ) || \pi_k -\pi ||_{TV}$
$|Z^n_k(y) - Z^n(y)| < f^n( y | \hat{\theta}_n ) || \pi_k -\pi ||_{TV}$ and
\begin{eqnarray}\label{eqn:boundkest}
\left| \frac{b_k}{Z^n_k(y)} - \frac{b}{Z^n(y)} \right| & < &
\left( \frac{b}{Z^n(y)} \frac{1}{Z^n(y)} + \frac{1}{Z^n(y)} \right)  f^{n}(y | \hat{\theta}_n) || \pi_k - \pi ||_{TV} \\
& \leq & \frac{2}{Z^n(y)}    f^{n}(y | \hat{\theta}_n) || \pi_k - \pi ||_{TV} ,\nonumber 
\end{eqnarray}
since $\left| \frac{b}{Z^n(y)} \right| \leq 1$,
and we obtain the result. The prove involving $\Q_y[\pi_k]$ and $\Q_y[\pi]$ is analogous.
\end{proof}

\begin{theorem}[Consistent rate of convergence] \label{teo:tv_rate}
With assumption \ref{assp:1}, the rate of convergence in (\ref{eqn:global_error2}) and $|| \pi_k - \pi ||_{TV} \rightarrow 0$ we have
$|| \Q_y^n[\pi_k] - \Q_y[\pi] ||_{TV} \rightarrow 0$ as $k,n \rightarrow \infty$ and
\begin{equation}\label{eqn:consistency}
|| \Q_y^n[\pi_k] - \Q_y[\pi] ||_{TV}  < 
 \frac{K_1}{Z_k(y)}  |\alpha(n)|^p   + \frac{f (y | \hat{\theta})}{Z(y)} || \pi_k - \pi ||_{TV}  \\
\end{equation}
for big enough $k$ and $n$ (note that $Z_k(y), Z(y) >0$ and $Z_k(y) \rightarrow Z(y)$).
\end{theorem}

\begin{proof}
Note that $|| \Q_y^n[\pi_k] - \Q_y[\pi] ||_{TV} = || \Q_y^n[\pi_k] - \Q_y^{k} + \Q_y^{k} - \Q_y[\pi] ||_{TV} \leq 
|| \Q_y^n[\pi_k] - \Q_y^{k} ||_{TV} + || \Q_y^{k} - \Q_y[\pi] ||_{TV}$ and from theorems \ref{teo:tvn_rate} and \ref{teo:tvk_rate}
we obtain the result.
\end{proof}

\begin{corollary}
With assumption \ref{assp:1}, the rate of convergence in (\ref{eqn:global_error2}) and $|| \pi_k - \pi ||_{TV} \rightarrow 0$ we have
$\Q_y^n[\pi_k] \Rightarrow \Q_y[\pi]$.
\end{corollary}

\subsection{Remarks on Theorems \ref{teo:tvn_rate}, \ref{teo:tvk_rate} and \ref{teo:tv_rate}}

\begin{itemize}

\item The ``posterior operator'' is Lipschitz continuos, that is
$$
|| \Q_y[\pi_k] - \Q_y[\pi]||_{TV} <   \frac{f (y | \hat{\theta})}{Z(y)} || \pi_k - \pi ||_{TV} .
$$

\item If the rate of convergence of the truncated prior $\pi_k$ to the complete prior $\pi$ is $|| \pi_k - \pi ||_{TV} < k^{-q}$
then, since $ [Z_k(y)]^{-1} \rightarrow [Z(y)]^{-1} > 0$,
$$
|| \Q_y^n[\pi_k] - \Q_y[\pi] ||_{TV} < K_2 n^{-p} + K_2' k^{-q}
$$
(with $|\alpha(n)|^{p} = O(n^p)$). That is, the discretized version of the posterior converges in total variation to
the theoretical posterior at the same rate as the FM and the prior truncation.

\item In many cases of PDE discretization schemes, the number of parameters or dimension of the prior $k$
increases (linearly, quadratically etc.) with the discretization size $n$ as it is the case
in some inverse problems using the Finite Element Method \cite[eg.][]{Bui-Thanh2013, Petra2014}.
In principle this should not represent an additional problem
and the consistency result in (\ref{eqn:consistency}) still holds for big enough
$n$ as far as $|| \pi_k - \pi ||_{TV} \rightarrow 0$.

\end{itemize}

\subsection{Posterior Estimates}

In modern Bayesian theory all inference problems are viewed in a perspective of a decision under uncertainty,
ultimately needing to maximize posterior expected utility,
which is in fact the Bayesian paradigm.  Moreover,
all utility functions are bounded and by convention normalized to $[0,1]$ \citep{DeGroot1970}.
If one wants to calculate the posterior expectation of an utility function, or any other bounded functional,
$h \in [0,1]$ note that
$$
| \hat{h}^{n,k} - \hat{h} | \leq || \Q_y^n[\pi_k] - \Q_y[\pi] ||_{TV}
<  \frac{K_1}{Z_k(y)}  |\alpha(n)|^p   + \frac{f (y | \hat{\theta})}{Z(y)} || \pi_k - \pi ||_{TV} .
$$
where $\hat{h}^{n,k} = \int h(\theta) \Q_y^n[\pi_k](d\theta)$ and $\hat{h} = \int h(\theta) \Q_y(d\theta)$.
That is, controlling $ || \Q_y^n[\pi_k] - \Q_y[\pi] ||_{TV}$ will bound the error in any estimation required and
the rates of convergence are transferred. (In passing, note from the prove of theorem \ref{teo:tvn_rate}, that is lemma \ref{lemm:conv_post1},
that $\frac{K_1}{Z_k(y)}  |\alpha(n)|^p$ is the bound for
$\frac{|Z^n_k(y) - Z_k(y)|}{Z_k(y)} = \left| \frac{Z^n_k(y)}{Z_k(y)} - 1 \right|$.)

Traditionally we are used to working with the posterior mean and/or variance.
In that case, $h$ is not bounded.  However, if $h$ is continuos and
the $\hat{h}^{n,k}$ are \textit{uniformly integrable}
then $\hat{h}$ exists and $\hat{h}^{n,k} \rightarrow \hat{h}$.  This can be verified if
\begin{equation}\label{eqn:uni_int}
\sup_{n,k} \int |h(\theta)|^{1+\epsilon} \Q_y^n[\pi_k](d\theta) < \infty .
\end{equation}
for some positive $\epsilon$ \citep[][chap. 2]{Billingsley1968}.  For example if the tails
of the finte dimensional posterior decay exponentially then
$s^2_{n,k}  = \int h^2(\theta) \Q_y^n[\pi_k](d\theta) < \infty$, needing only to verify that
these $s^2_{n,k}$ are bounded.


\section{Expected a priori bounds and Bayes Factors}\label{sec:EABF}

As in \cite{Capistran2016} in order to find reasonable guidelines to choose
a discretization level $n$ and a suitable prior truncation $k$, we compare the numeric posterior $\Q_y^n[\pi_k]$
with the theoretical posterior $\Q_y[\pi]$ using Bayesian model selection, namely Bayes Factors (BF).
Assuming an equal prior probability for both models, the BF is the posterior odds of one model against the other,
that is $\frac{p}{1-p}$ where $p=\frac{ Z^n_k(y) }{ Z^n_k(y) +  Z(y) }$, the posterior probability
of the numerical model.  That is, the BF is the ratio of the normalization constants $\frac{Z^n_k(y)}{Z(y)}$.
 In terms of model equivalence
an alternative expression conveying the same odds is
$$
\frac{1}{2} \left| \frac{Z^n_k(y)}{Z(y)} - 1 \right| .
$$
We now try to control the Bayes Factor between the discretized model and the theoretical model,
$\frac{Z^n_k(y)}{Z(y)}$, through the use of the Absolute BF (ABF).
In order to do that, independently of the specific data at hand, we try to bound the expected ABF (the EABF),
$$
\int \frac{1}{2} \left| Z^n_k(y) - Z(y) \right| \lambda(dy) =\int \frac{1}{2} \left| \frac{Z^n_k(y)}{Z(y)} - 1 \right| Z(y) \lambda(dy) ,
$$
in terms of estimates on the error in the numeric forward map, as in (\ref{eqn:global_error1}).  The idea is
to keep the EABF below a small threshold (eg. $\frac{1}{20}$) so that the BF is close to 1 and the difference between
the numeric and the theoretical model is ``not worth more than a bare mention''  \citep{KASS1995, Jeffreys61}.

\begin{theorem}\label{teo:eabf1}
With assumption \ref{assp:1},
the rate of convergence in (\ref{eqn:global_error2}), $|| \pi_k -\pi ||_{TV} \rightarrow 0$
and $\phi_y( \eta ) = -\log f_o(y | \eta ) \in C^1$ $\lambda$-a.s. we have
\begin{eqnarray}\label{eqn:EABF1}
& & \int \frac{1}{2} \left| \frac{Z^n_k(y)}{Z(y)} - 1 \right| Z(y) \lambda(dy) < \\ \nonumber
& & \frac{K_0 |\alpha(n)|^{p}}{2} \sum_{i=1}^m
\int \int \left| \frac{\partial}{\partial \eta_i} \phi_y( \H(\F^{\alpha(n)}(\theta)) ) \right|
f_o( y | \H(\F^{\alpha(n)}(\theta))) \lambda(dy) \pi_k(d\theta) \\
& &+ || \pi_k -\pi ||_{TV} \nonumber.
\end{eqnarray}
\end{theorem}

\begin{proof}
As seen in the proof of theorem \ref{teo:tvk_rate} we have
$$
|Z_k(y) - Z(y)|  = 
\left| \int f( y | \theta) (\pi_k - \pi)(d\theta) \right| \\
\leq \int f (y | \theta) |\pi_k -\pi |(d\theta) 
$$
and therefore $\int |Z_k(y) - Z(y)| \lambda(dy)  \leq 
\int \int f (y | \theta) |\pi_k -\pi |(d\theta) \lambda(dy) = \int \int  f (y | \theta) \lambda(dy) \allowbreak |\pi_k -\pi |(d\theta) =
2 || \pi_k -\pi ||_{TV}$. Therefore
\begin{eqnarray*}
& & \int \frac{1}{2} \left| \frac{Z^n_k(y)}{Z(y)} - 1 \right| Z(y) \lambda(dy) \leq \\
& & \int \frac{1}{2} \left| Z^n_k(y) - Z_k(y) \right| \lambda(dy) + || \pi_k -\pi ||_{TV}.
\end{eqnarray*}
To bound the last integral, note that
$$
\left| Z^n_k(y) - Z_k(y) \right| = \left| \int f( y | \theta)(R_n(\theta) -1) \pi_k(d\theta) \right|;
R_n(\theta) = \frac{f^n(y | \theta)}{f(y | \theta)} .
$$
For $\eta$ close enough to $\eta_1$, the a likelihood ratio $\frac{f_o( y | \eta)}{f_o( y | \eta_1)}$ is near to 1 and
$$
\left| \frac{f_o( y | \eta)}{f_o( y | \eta_1)} - 1 \right| \cong \left| \log \left( \frac{f_o( y | \eta)}{f_o( y | \eta_1)} \right) \right|
= \left| \phi_y( \eta ) - \phi_y( \eta_1 ) \right| = | \phi_y(\eta_1)-\phi_y(\eta)| 
$$
With the first order Taylor approximation of $\phi_y(\eta)$ around $\eta_1$ we have 
$$
| R_n(\theta) -1| =  | \phi_y(\eta_1)-\phi_y(\eta)| =\left | \nabla \phi_y ( \H(\F^{\alpha(n)}(\theta)) ) \cdot ( \H(\F^{\alpha(n)}(\theta)) -  \H(\F(\theta))) + R \right | .
$$
Ignoring the higher order terms in the residual and using the error bound in (\ref{eqn:global_error1}) we have
\begin{eqnarray*}
& & \int \frac{1}{2} \left| Z^n_k(y) - Z^n(y) \right| \lambda(dy) < \\
& & \frac{K_0 |\alpha(n)|^{p}}{2} \int \int f_o( y | \eta) || \nabla \phi_y ( \H(\F^{\alpha(n)}(\theta)) ) ||_1 \pi_k(d\theta) \lambda(dy) ,
\end{eqnarray*}
since for any two vectors $|a \cdot b| = |\sum a_i b_i| < c \sum | b_i |$ with $|a_i| < c$ and we obtain the result.
\end{proof}

We may attempt to calculate the remaining double integral by changing the order of integration letting
$\int M( \H(\F(\theta)) ) \pi_k(d\theta) $ and
\begin{equation}\label{eqn:int_eta}
M(\eta) = \int || \nabla \phi_y ( \eta ) ||_1 f_o( y | \eta) \lambda(dy) =
\sum_{i=1}^m \int \left| \frac{\partial}{\partial \eta_i} \phi_y( \eta ) \right| f_o( y | \eta) \lambda(dy) .
\end{equation}
This in general is difficult to achieve, however it is possible if it happens that $M(\eta)$ does not depend
on $\theta$.

In the usual case of independent Gaussian errors with known variance $\sigma^2$,
$|| \nabla \phi_y ( \eta ) ||_1 \allowbreak = \sigma^{-1} \sum_{i=1}^m \left| \frac{y_i - \eta_i}{\sigma} \right|$ and
$
\sum_{i=1}^m \int \left| \frac{y_i - \eta_i}{\sigma} \right| N( y_i | \eta_i, \sigma) dy_i = \sqrt{\frac{2}{\pi}} \frac{m}{\sigma},
$
since $\int |x| \frac{1}{\sqrt{2\pi}} e^{-\frac{x^2}{2}} dx = \sqrt{\frac{2}{\pi}}$.
This result may be generalized to any location-scale family and we present it next. 

\begin{theorem}\label{teo:eabf2}
With the setting of theorem \ref{teo:eabf1}, assuming independent data arising from a location-scale family,
namely
$$
f_o( y | \eta) = \prod_{i=1}^m \sigma^{-1} \rho\left( \frac{y_i - \eta_i}{\sigma} \right)
$$
with $\rho$ a bounded $C^1$ symmetric Lebesgue density in $\R$ with $\int_{-\infty}^{\infty} x^2 \rho(x) dx = 1$
then
\begin{equation}\label{eqn:EABF2}
\int \frac{1}{2} \left| \frac{Z^n_k(y)}{Z(y)} - 1 \right| Z(y) \lambda(dy) <
\rho(0) \frac{K_0 |\alpha(n)|^{p}}{\sigma}  m +  || \pi_k -\pi ||_{TV} .
\end{equation}
\end{theorem}

\begin{proof}
From (\ref{eqn:int_eta}) note that
$$
\int \left| \frac{\partial}{\partial \eta_i} \phi_y( \eta ) \right| f_o( y | \eta) \lambda(dy) = 
\int_{-\infty}^{\infty} \left| \sigma^{-1} V'\left( \frac{y_i - \eta_i}{\sigma} \right) \right| \sigma^{-1}
\rho\left( \frac{y_i - \eta_i}{\sigma} \right) dy_i
$$
where $\rho(x) = e^{V(x)}$.  The integral on the rhs is in fact equal to
$2 \sigma^{-1} \int_{0}^{\infty} V'(x) \rho(x) dx = 2 \sigma^{-1} \rho(0)$ (since $\rho'(x) = V'(x) \rho(x)$),
and we obtain the result.
\end{proof}

Since $K_0 |\alpha(n)|^{p}$ is the error in the FM (with the observation operator in (\ref{eqn:global_error1})),
measured in the same units as the $y_j$s, note from (\ref{eqn:EABF2}) that $\frac{K_0 |\alpha(n)|^{p}}{\sigma}$ is the relative error in the numeric FM  with respect to the standard error in the observations $\sigma$.
In order to keep the EABF below a threshold we require more precision in the FM if the sample size $m$
increases and more (less) precision in the FM if the standard error decreases (increases).
It makes much sense to measure $K_0 |\alpha(n)|^{p}$ with respect to $\sigma$ and $\frac{K_0 |\alpha(n)|^{p}}{\sigma}$ becomes
units free.

If we let the $EABF \leq b$, and for example  $b= \frac{1}{20} = 0.05$,
we expect nearly no difference in the numerical and the
theoretical posterior.  If we set the error in the FM $K = K_0 |\alpha(n)|^{p}$ then we require
$\rho(0) \frac{K}{\sigma}  m +  || \pi_k -\pi ||_{TV} < b $, that is,
we need the numerical error in the FM in (\ref{eqn:global_error1})
\begin{equation}\label{eqn:main_bound}
K < \frac{\sigma}{m} \frac{b - || \pi_k -\pi ||_{TV}}{\rho(0)} .
\end{equation}
We require $|| \pi_k -\pi ||_{TV} < b$, but since this only involves the prior truncation
we should be able to fix it from the onset.  For example, $|| \pi_k -\pi ||_{TV} < \frac{1}{100}$.

Our suggested procedure is to run the solver, including an after the fact error estimate
(or \textit{a posteriori} error estimate, we use \textit{after the fact} given the conflict of terms with the Bayesian jargon).
If the error in the FM does not comply with the bound in (\ref{eqn:main_bound}), then run the solver again
with a finer discretization $\alpha(n)$.  In passing, we assure (\ref{eqn:global_error1}) for all $\theta \in \Theta$.
Note that in ODEs the RK45 method \citep[Rungue-Kutta order 5 method of][for example]{Cash1990} produces
after the fact error estimates.  More recently, the discontinuos Galerking  method for PDEs may include
high order solvers with after the fact error estimates \citep{di2011mathematical, hesthaven2007nodal}.
In general, error estimates for PDEs are much harder to obtain and the usual strategy is to
consider adjoint-base methods.

\section{Using a base for $\Theta$}\label{sec:base}

Defining a prior directly on the Banach space $\Theta$ is difficult and we have little options, as for example
an infinite dimension Gaussian distribution \citep{Stuart2010}.  A perhaps more pragmatic approach is to decide
on a base for $\Theta$ to represent its elements, and then take the coefficients in the base representation as random,
as in \cite{Scheichl2017}.  Accordingly, let $\Theta = C(D)$ be the continuous functions on a compact domain
$D \subset \R$ with norm $|| \cdot ||$ which can be $L_2$ for example.  This indeed constitutes a
separable Banach space.  Let, for any $\theta \in \Theta$
\begin{equation}\label{eqn:base_exp}
\theta(t) = \theta_0(t) + \sum_{i=1}^\kappa \beta_i \phi_i(t) ,
\end{equation}
where $\phi_i$ are our chosen base, $\beta_i \in \R$ and $\theta_o \in \Theta$ is fixed.  We take the base functions
normalized $|| \phi_i || = 1$.  Let $\kappa$ be a discrete random variable and $\beta_1, \beta_2, \ldots$ be random variables in $\R$,
then a probability measure on $\N \times \R^{\infty}$ defines the distribution $F$ of $(\kappa, \beta_1, \beta_2, \ldots )$
and the prior distribution $\pi$ will be the push forward measure over the function
$$
\kappa, \beta_1, \beta_2, \ldots    \mapsto^g \theta_0(t) + \sum_{i=1}^\kappa \beta_i \phi_i(t) .
$$
The marginal distribution $F_k$ of the first $k$ terms, which is its $k$th natural projection,
defines the push forward measure $\pi_k$ from
$\theta_k(t) = \theta_0(t) + \sum_{i=1}^{\min(k, \kappa)} \beta_i \phi_i(t)$,
 which is our truncated approximate prior.

With lemma 2.1 of \cite{Rosalsky1997}, on convergence of random elements in Banach spaces,
we have that if
$$
\sum_{i=1}^\infty E[||\beta_i \phi_i ||] = \sum_{i=1}^\infty E|\beta_i| < \infty
$$
then there exists $\theta(t) \in \Theta$ such that
$$
\theta_0(t) + \sum_{i=1}^k \beta_i \phi_i(t) \rightarrow \theta(t) ~\text{a.s.}
$$
This implies $\theta_k  \rightarrow \theta$ in probability and therefore $\theta_k  \Rightarrow \theta$.
Since $F_k \Rightarrow F$ (the $F_k$s are the finte dimensional marginals) and $g$ is continuous,
by the mapping theorem it also implies $\pi_k \Rightarrow \pi$  \citep{Billingsley1968}.

With this we have $\Q_y^n[\pi_k] \Rightarrow \Q_y^n[\pi]$ and
$\Q_y[\pi_k] \Rightarrow \Q_y[\pi]$ as in theorem \ref{teo:tvk_rate} and note that so far the $\beta_i$s need not be independent.
The only requirement here is
\begin{equation}\label{eqn:assumption2}
\sum_{i=1}^\infty E|\beta_i| < \infty .
\end{equation}


To control the rate of convergence we requiere convergence in Total Variation.
From the coupling characteristic of the Total Variation norm  \citep{Gibbs2002} 
$|| \pi_k -\pi ||_{TV} \leq P(\theta_k \neq \theta)$ and therefore we have
\begin{equation}\label{eqn:TV_bound}
|| \pi_k -\pi ||_{TV} \leq P(\kappa > k) .
\end{equation}
Let $h(\cdot)$ be the prior for $\kappa$, then $P(\kappa > k) = \sum_{i=k+1}^\infty h(i)$.

A typical choice for $h(i)$ would be a Poisson distribution with parameter $\lambda$ then
$\sum_{i=k+1}^\infty e^{-\lambda} \frac{\lambda^i}{i !}$. 
For example, if a priori the average number of terms in (\ref{eqn:base_exp}) is $\lambda=10$ then
with $k=20$, $\sum_{i=k+1}^\infty e^{-\lambda} \frac{\lambda^i}{i !} < \frac{1}{100}$.

From (\ref{eqn:EABF2}) we see that the overall EABF bound in this case is
\begin{equation}\label{eqn:EABF3}
\int \frac{1}{2} \left| \frac{Z^n_k(y)}{Z(y)} - 1 \right| Z(y) \lambda(dy) <
\rho(0) \frac{K}{\sigma} m + \sum_{i=k+1}^\infty h(i) 
\end{equation}
with $||  \H(\F(\theta)) - \H(\F^{\alpha(n)}(\theta)) || < K$.

\subsection{The discretized numeric posterior}

To be able to work on our posterior distribution we need to truncate the prior of $\kappa$ below
some maximum $k$, thus implicitly truncating the prior $\pi$ to $\pi_k$.
At the end we are left to deal with the varying dimensional posterior, with maximum dimension $k$
\begin{eqnarray*}
\pi( \beta_1, \ldots, \beta_l, l | y) & \propto & \\
&& \sigma^{-m}
\prod_{j=1}^m \rho \left( \frac{y_j - \H_j(\F^{\alpha(n)}( \theta_0(t) + \sum_{i=1}^l \beta_i \phi_i(t) ))}{\sigma} \right) \\
&& 1(l \leq k) h(l) ,
\end{eqnarray*}
subject to $||  \H(\F(\theta)) - \H(\F^{\alpha(n)}(\theta)) || < K = \frac{\sigma}{m} \frac{b - \sum_{i=k+1}^\infty h(i)}{\rho(0)}$ (eg. $b = \frac{1}{20}$).

At this point we have two options, we may work with the full model with $l=k$ or take $l$
also as a parameter to be inferred.  The latter has the great advantage in that the posterior
will select the ``effective dimension'' \cite{Palafox2015} of our model although is far more computational demanding than the former.  For the sheer complexity of the FMs, we leave $l=k$ fixed
in examples~\ref{sec:1d_pois} and~\ref{sec:2d_pois}.

When $l$ is also a parameter we may run an MCMC for each $l \leq k$.  The posterior probability 
of each $l$ can be obtained estimating the normalization constant given $l$.
This is a difficult estimation processes \citep{Valpine2008, Palafox2015},
but in some cases of near Gaussian
posteriors the normalization constants are easier to obtain; this approach is used in
example~\ref{sec:hugo}.

A different approach is to use a transdimensional MCMC (as RJMCMC) to include
$l$ in the MCMC process.  This we do in example~\ref{sec:deconv}.

\section{Examples}\label{sec:examples}

We first review some representative Bayesian UQ examples that recently appeared in
the literature and briefly view them in the perspective of our results.  Second in
sections \ref{sec:hugo}, \ref{sec:deconv}, \ref{sec:1d_pois} and \ref{sec:2d_pois} we present workout
examples considering Bayesian UQ problems for a 1D wave equation,
deconvolution and 1D and 2D heat equations, respectively.

\bigskip
EXAMPLE 1: In \cite{Lassas2004} and \cite{Kolehmainen2012} the parameter space $\Theta$ is the space of continuous functions in
the unit interval $C[0,1]$.  For piecewise linear continuos functions on $[0,1]$ the ``total variation'' prior
is proposed to be used, for a discretization $k$
$$ 
\pi_k(u) = c_k \exp\left\{ - \alpha_k \sum_{j=1}^{k+1} |u_j^k - u_{j-1}^k | \right\} ,
$$
where $u(t) = u_{j-1}^k + (t - t_{j-1}) \frac{u_j^k - u_{j-1}^k}{t_j - t_{j-1}}; t \in (t_{j-1},t_j]$ and $t_j = \frac{j}{k+1}$.
Inconsistencies are found in the MAP and CM estimators (the maximum of the posterior and the posterior mean)
when $\alpha_k = 1$ or $\alpha_k = \sqrt{k + 1}$ and $k \rightarrow \infty$.  

A clear problem with this approach is that we do not know which is the prior $\pi$ on $C[0,1]$, what is the measurable
space and how $\pi_k$ converges to $\pi$, if at all converges.
How can we expect consistency without the latter? 
Defining a probability measure on $C[0,1]$ is a complex and delicate endeavour \citep[][chap. 2]{Billingsley1968}
and is indeed a source of classic results in probability (eg. the Weiner process is a measure on $C[0,1]$).

\bigskip\noindent
EXAMPLE 2:  \cite{Scheichl2017} worked with a continuos FM with Gaussian errors,
derived from an elliptic PDE, with error bounds equivalent to (\ref{eqn:global_error1}).
The posterior is needed to be defined in a functional space, a separable Banach space.
This is sufficient for
Assumption \ref{assp:1} to hold.
They use a base expansion as in (\ref{eqn:base_exp}) with independent and
summable $\beta_i$s.  The latter is sufficient for weak convergence, beyond their specific prior
for the $\beta_i$s.  Therefore,  the results in section~\ref{sec:weak} apply.

\bigskip\noindent
EXAMPLE 3: \cite{Christen2017} considered a two dimensional inverse problem of
the logistic ODE. The FM is indeed continuos, seen from the analytic solution
$X(t) = \frac{KX_0}{X_0 + (K-X_0)e^{- r t}}$.  They consider Gaussian errors
and a Rungue-Kutta method of order 5, with error bounds similar to
(\ref{eqn:global_error1}).  Therefore, lemmas \ref{lem:joint} and \ref{lem:bayes} and
consistency theorems \ref{teo:weak1} and \ref{teo:tvn_rate} apply.

They used a RK45 to solve the ODE and obtain error estimates; these were larger
than the actual errors also available from comparison from the analytic solution.  The bound
for the numeric solver in (\ref{eqn:main_bound}) is kept adaptively for EABF $< \frac{1}{20}$
(no prior truncation is needed),  and also a fine grid solver was use.  The adaptive solver gave
posterior distributions basically indistinguishable to those obtained by the fine solver, with
more than 90\% CPU time save.

\bigskip\noindent
EXAMPLE 4:  \cite{Christen2017} also considered a FM arising from the Burgers' PDE
in a two dimensional Bayesian inverse problem with Gaussian errors.
The FM is indeed continuos, seen again from the analytic solution
The authors used a second-order accurate finite-volume solver with error bounds as in
(\ref{eqn:global_error1}).  Again lemmas \ref{lem:joint} and \ref{lem:bayes} and
consistency theorems \ref{teo:weak1} and \ref{teo:tvn_rate} apply.

More importantly, they kept adaptively EABF $< \frac{1}{20}$ and compared with a finer
solver obtaining a 60\% save in CPU time.  The resulting posteriors where indistinguishable
for all practical purposes.

\bigskip\noindent
EXAMPLE 5: In \cite{Capistran2012} and inverse problem in epidemics, driven by a system of ODEs,
is analyzed with a Generalized Discrete distribution model for the data
\citep[a combination of Binomial, Poisson and Negative-Binomial distributions, see][]{Capistran2011}.
This discrete family can be seen to produce continuos likelihoods for independent data and
since these are pmf's the likelihood is always below or equal to 1.  Using standard results
on the continuity of solutions of ODE over parameters the FM may be proved to be continuos.

\bigskip\noindent
EXAMPLE 6: \cite{Bui-Thanh2013} and \cite{Petra2014} work with an infinite dimensional Bayesian inverse problem,
using a Gaussian prior in a $L^2$ functional space with possibly correlated Gaussian data.  The FM is assumed continuous and
therefore the existence lemmas \ref{lem:joint} and \ref{lem:bayes} apply.  The authors suggest using a Langrange basis
functions to represent the elements of $\Theta$ as in (\ref{eqn:base_exp}) in a Finite Element discretization of the FM.
However, the authors do not discuss
the a priori convergence of $\sum_{i=1}^\infty E[||\beta_i \phi_i ||]$ therefore the results of section \ref{sec:base} cannot be
applied directly.  This is an example where the number of parameters represents the prior truncation $k$ and this increases with the discretization size $n$. 

We now present 4 workout examples.  In all cases we consider Gaussian noise for the observations with known
stadard error, as in (\ref{eqn:loc_scale_model}), and therefore the only relevant part to be taken care for in
assumption \ref{assp:1} is that the theoretical and the numeric FM are continuos, in order for the
corresponding posteriors to be correctly defined.

As far as the derivation of the EABF bound is concerned, we require that the numeric
FM error bound in (\ref{eqn:global_error1}) exists for all $\theta$.

\subsection{A 1D wave equation example}\label{sec:hugo}

Consider the homogeneous Dirichlet conditions for the wave equation
\begin{equation}\label{enq:Wave Dirichlet}
u_{tt}=c^2u_{xx},\qquad x\in(0,l),  \qquad u(0,t) = 0 = u(l,t)
\end{equation} 
with initial conditions $u(x,0)=\phi(x), u_t(x,0)=\psi (x)$.

Under a separation of variables technique, a solution of the above problem can be found substituting $u(x,t)=X(x)T(t)$ in the PDE. This problem becomes a pair of separate ordinary differential equations for $X(x)$ and $T(t)$ given by
\begin{equation}\label{eq:ODEs}
X''+\beta^2 X=0  \qquad  \text{ and } \qquad T''+c^2\beta^2 T=0
\end{equation}
With $\phi(x)= \sum_{n}A_n \sin\left(\dfrac{n\pi x}{l}\right)$ and
$\psi(x)= \sum_{n}\dfrac{n\pi c}{l}B_n \sin\left(\dfrac{n\pi x}{l}\right)$ we obtain 
\begin{equation}\label{eq:Solseries}
u(x,t)=\sum_{n}\left(A_n\cos \left(\dfrac{n\pi ct}{l}\right)+B_n\sin \left(\dfrac{n\pi ct}{l}\right)\right)\sin\left(\dfrac{n\pi x}{l}\right)
\end{equation}

To simplify the computations, let us consider the case $c=1$ and $\psi(x)=0$, that is, $B_n=0$ for all $n$. Therefore $u(x,t)=\sum_{n}A_n \cos \left(n\pi t\right) \sin \left(n\pi x \right)$ where $A_n =2\int_{0}^1 \phi(x) \sin(n\pi x)dx$ and $t=1$ then
\begin{equation}\label{eq:TSolEx}
u\left(x,1\right)=\sum_{n=0}^{\infty}A_n \left(-1\right)^{n}\sin\left(n\pi x\right).
\end{equation}

The inverse inference problem is as follows.
Given measurements of $u(x,1)$ at $z_0,z_1,\ldots,z_m \in (0,1)$, we need to infer
the unknown function $\phi(x)$.  Namely, consider the case 
\begin{equation}\label{eq:Obs}
y_j=u(z_j,1)+\varepsilon_j, \qquad j=1,2,...,m
\end{equation}
where $\varepsilon_j\sim \mathcal{N}(0,\sigma^2)$.

In this case we consider the FM and the observation functional
as the identity, $\H\F[\theta] = \theta = \sum_{n=0}^{\infty} A_n \phi_n(x)$,
with $\phi_n(x) = \left(-1\right)^{n}\sin\left(n\pi x\right)$.  No error is considered
in the FM and only a truncation $\kappa$ is considered in the series, that is
$\theta_k = \sum_{n=0}^{\kappa} A_n \phi_n(x)$.  Evidently the FM
is continuos and regularity conditions are met for the infinite dimension
posterior to exists.  Regarding the bound in (\ref{eqn:main_bound}) only
the $|| \pi - \pi_k ||_{TV}$ term is relevant since the error bound for the FM is zero.
That is, to bound the EABF we only need to bound the a priori truncation
error $|| \pi - \pi_k ||_{TV}$, which we do below.  In this case, the marginal posterior distribution of $\kappa$
is simple to calculate, since normalizations constants are available analytically, to
obtaining the effective dimension of the problem.

Since the FM is linear may therefore express (\ref{eq:Obs}) as a linear model in the usual way, namely
\begin{equation}\label{eq:LinearObs}
y = X_\kappa \beta_\kappa + \varepsilon
\end{equation}
where $\beta_\kappa = \left(A_1,A_2,\ldots, A_\kappa\right)$ and $X_\kappa$ a $m\times \kappa$
matrix where each row of $X_\kappa$ is 
$$
\left( -\sin(\pi z_j), \sin(2\pi z_j), \ldots, (-1)^\kappa\sin(\kappa\pi z_j)\right) .
$$

As in section~\ref{sec:base} a priori $\kappa \sim h(\cdot)$ and a truncated prior is obtained by restricting
$\kappa < k$.  Considering a priori
$\beta_\kappa \sim \mathcal{N}\left( \mu_0^\kappa, \sigma^2 (A_0^\kappa)^{-1} \right)$, given $\kappa$ the
posterior for $\beta_\kappa$ is
$$
\beta_\kappa | \kappa, y \sim \mathcal{N}\left((X_\kappa^TX_\kappa + A_0^\kappa)^{-1}(A_0^\kappa\mu_0^\kappa+X_\kappa^Ty),\sigma^2(X_\kappa^TX_\kappa+A_0^\kappa)^{-1}\right) .
$$
The normalization constant for these models are readily available, to obtain the marginal posterior
distribution for $\kappa$, namely
\begin{equation}\label{eqn:kappa}
P( \kappa = i | y) \propto h(i)I(i < k) \frac{|A_0^i|^{1/2}}{| X_i^TX_i + A_0^i |^{1/2}}
\exp \left\{ \dfrac{1}{2\sigma^2}y^TX_i\left(X_i^TX_i+A_0^i\right)^{-1}X_i^Ty \right\} .
\end{equation}

Synthetic data was obtained with $\sigma=0.025, m=15$ with the true
$\phi(x)=1.5\sin(\pi x)+0.8\sin(2\pi x)+0.7\sin( 3\pi x)+0.3 \sin(4\pi x)$, that is $\kappa=4$.
The prior $h(\cdot)$ for $\kappa$ is a $Po(10)$.  In figure~\ref{fig:Wave_1D} we present
$P( \kappa = i | y)$  truncated to $\kappa \leq 15$; note that $\sum_{i=k+1}^\infty h(i) < \frac{1}{20}$
already to bound the EABF accordingly.  Additionally we produced $P( \kappa = i | y)$ renormalizing it with $\kappa \leq 20$,
obtaining virtually the same results (not shown).  In fact, summing up the normalization constants
in (\ref{eqn:kappa}) provides $Z_k(y)$ and summing up to 20 provides an estimate
of $Z(y)$ from which we can produce an estimate of the ABF
$\frac{1}{2}\left| \frac{Z_k(y)}{Zy)} -1 \right|$ which results in $1.3\times 10^{-10}$, very well below
$\frac{1}{20}$.
 
\begin{figure}
\begin{center}
\includegraphics[scale=0.4]{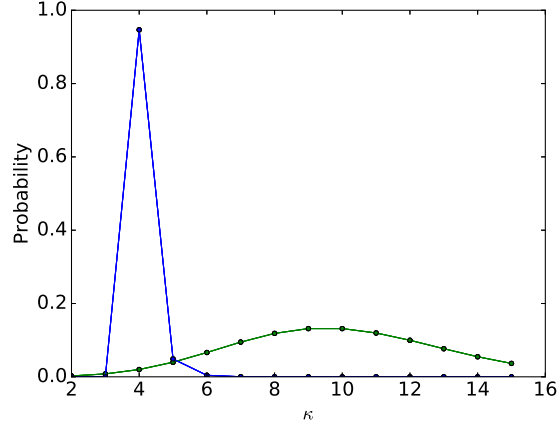}
\caption{\label{fig:Wave_1D} The marginal posterior pmf of $\kappa$, the parameter dimension,
and the prior of $\kappa$ (green), truncated to $\kappa \leq 15$.  The true dimension
is 4, which corresponds to the map of this posterior.  
}
\end{center}
\end{figure}

\subsection{A deconvolution example}\label{sec:deconv}

We present a 1D deconvolution example where
an exact solution is available and Simpson's rule is used to also have a numeric
version of the FM.  Here we illustrate both, a numeric FM with a discretization and
a truncation in the prior.  The bound in (\ref{eqn:main_bound}) is used to bound
the EABF obtaining nearly identical results as using the exact FM, in a trans dimensional
MCMC, to also obtain the marginal posterior for $\kappa$. 

We consider the convolution of $\theta$ with the kernel $c$
\begin{equation}\label{eqn:conv}
\F[\theta] = \int_0^1 c(y-x)\theta(y)dy
\end{equation}
which constitutes de FM.  Assume $c(z) = \frac{1}{2\alpha} 1_{[-\alpha,\alpha]}(z)$ and
$\theta(t) = \beta_0+\displaystyle\sum_{i=1}^{\infty} \beta_i \cos(2\pi i t) + \alpha_i \sin(2\pi i t)$.
With the $L_2$ norm the base functions have constant norm (independent of $i$)
equal to $\frac{1}{\sqrt{2}}$, we do not multiply by $\sqrt{2}$.

With the change of variable $u=y-x$ and identifying correctly the indicator function,
(\ref{eqn:conv}) may be calculated with
$\int_{\max\{x-\alpha , 0\}}^{\min\{x+\alpha , 1\}}\frac{1}{2\alpha} \theta(z) dz$;
this integral may be calculated analytically for each base function $\cos(2\pi i t)$ or $\sin(2\pi i t)$ in the series
definition of $\theta$.  Therefore, for a truncated series given $\kappa$,
$\theta_{\kappa} (t) = \beta_0+\displaystyle\sum_{i=1}^{\kappa} \beta_i \cos(2\pi i t) + \alpha_i \sin(2\pi i t)$, and
$\F[\theta_{\kappa}]$ is available analytically.  To construct a numerically defined FM $\F^n[\theta_{\kappa}]$ we use
Simpson's rule with a grid of size $n$ to evaluate the integral
$\int_{\max\{x-\alpha , 0\}}^{\min\{x+\alpha , 1\}}\frac{1}{2\alpha} \theta_{\kappa}(z) dz$.


The deconvolution inverse problem arises for the case when there are observations available from
the convolution, ie. $\theta$ is unknown and one wants to infer $\theta$.  That is
$$
y_i = \F[\theta](t_i) + \sigma \epsilon_i; \epsilon_j \sim N(0,1),
$$
$0 = t_1 < t_2 < \ldots < t_m = 1$ evenly spaced observation points (in this case the observation
functional $\H$ is the identity).  The error in the FM is calculated directly with
$|\F[\theta_{\kappa}](t_i) - \F^n[\theta](t_i)|$ since in this example the theoretical FM $\F[\theta_{\kappa}]$
is also available.  
The parameters needed to be inferred are $\beta_0, \beta_1, \alpha_1,  \beta_2, \alpha_2, \ldots$ .
A priori, an independent truncated normal prior in $[-a,a]$ with mean 0, $TN_a(s) =  \frac{1}{\sqrt{2 \pi} s (1-2\Phi(-a s))}
\exp(-\frac{1}{2} \frac{x^2}{s^2}) I_{[-a,a]}(x)$, is assigned to each $\beta_i, \alpha_i$ such that
$$
\beta_0 \sim TN_a( s_0 = \sigma_\beta ) ~~\text{and}~~
\beta_i, \alpha_i \sim TN_a( s_i = \sigma_\beta e^{-(i-1) \lambda_\beta}); i=1,2, \ldots .
$$

Evidently $\F[\theta_{\kappa}]$ and $\F^n[\theta]$ are continuos.
A global error bound, as in (\ref{eqn:global_error1}), is indeed sought, for all $\theta$s, since the support for
the $\beta_j$s and $\alpha_j$s is compact.  Since $\sum s_i$ is convergent, the sine-cosine series converges and
the prior distribution on the $\beta_j$s and $\alpha_j$s induces a prior $\pi$ for $\theta$, as explained in
section~\ref{sec:base}.  For the prior for $\kappa$ we take a Poisson with mean 8, but shifted to 1 and
renormalized to odd numbers, so $\kappa=1,3,\ldots$ only.
Truncating this prior to $\kappa < k$ terms induces the truncated prior $\pi_k$, as explained in section~\ref{sec:base}.

We produce $m=10$ synthetic data points with $\sigma = 0.02$, taking as the true $\theta$ the sine-cosine series function
with coefficients $\beta_0 = 0.9,  \beta_1=\alpha_1=-0.4,  \beta_2=\alpha_2=-0.3, \beta_3=\alpha_3=-0.2, \beta_i=\alpha_i=0; i\geq4$.  That is, the true dimension is $\kappa=7$.  The true sine-cosine series function, its convolution and the simulated
data points may be seen in figure~\ref{fig:DeConv_DataDim}.

For the prior we let $a=1$,
$\sigma_\beta = 0.3$ and $\lambda_\beta = -(1.0/10) * \log(0.1)$,
so that $\beta_{10}$ has 0.1 of the std. dev. of $\beta_1$.  The truncated normals are well
contained in the $[-a,a]$ interval.  
The posterior is truncated at dimension $k=12$, so that the tail of the Poisson prior is less than
0.01 leading to $||\pi - \pi_k|| < 0.01$, as explained in section~\ref{sec:base}. 

We designed a RJMCMC, using the t-walk (an affine invariant MCMC) within each
dimension.  The transdimensional jump move is simple, proposing a new
$\beta^p_{j+1} \sim N(  0, \beta_j/4)$ (centered at cero with a smaller size than the previous
$\beta$) and equivalently for the $\alpha_j$s.

We ran our RJMCMC with the approximate FM with errors complying with the
bound in (\ref{eqn:main_bound}), $K < \frac{\sigma}{m} \frac{b - || \pi_k -\pi ||_{TV}}{\rho(0)}$,
taking $b=\frac{1}{20}$.  In this case $\rho(0) = \frac{1}{\sqrt{2 \pi}}$ since we
are considering Gaussian errors.  We also ran our RJMCMC with the exact
FM for comparisons.

The t-walk mixes quite well in each dimension and with an Integrated Autocorrelation
Time of around 120.  We took 1,000,000 iterations of the RJMCMC, with a burn-in of
1,000, leading to an effective sample size of roughly 8,000.  This is good enough 
to create a histogram for
parameters up to $\beta_4$ and $\alpha_4$  (dimension = 9, see figure~\ref{fig:DeConv_DataDim}(b)).
Higher dimensions are seldom visited and the corresponding effective sample for
$\beta_5$ and $\alpha_5$ and above is very small, even for 1,000,000 iterations, leading to high Monte Carlo errors.
The posterior probability for each dimension is shown in figure~\ref{fig:DeConv_DataDim}(b) and
the corresponding posterior marginals are shown
in figure~\ref{fig:DeConv_Posts}.  In this, since we use a MC approach no estimation of the ABF is
readily available.

\begin{figure}
\begin{tabular}{c c}
\includegraphics[scale=0.35]{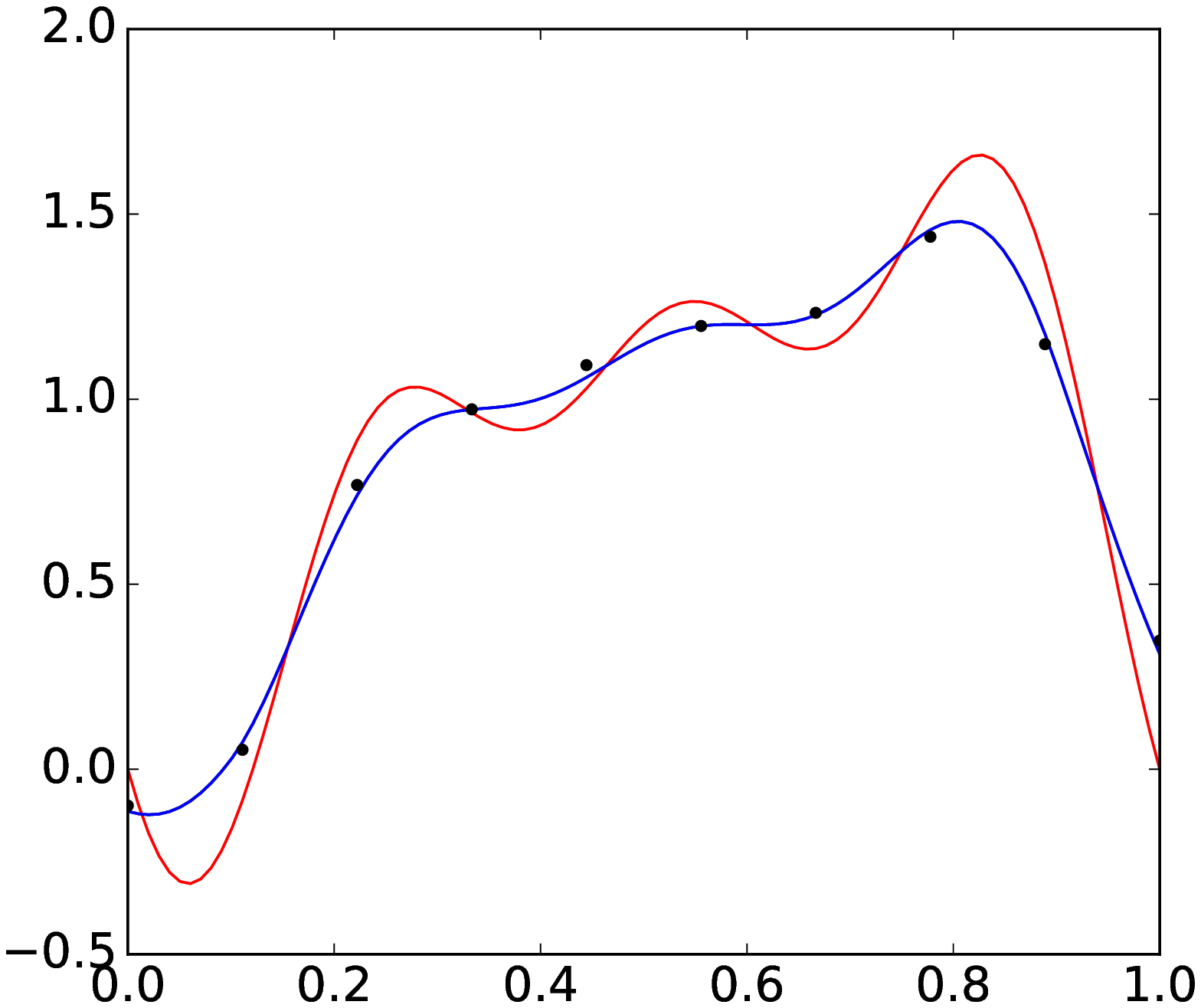} &
\includegraphics[scale=0.35]{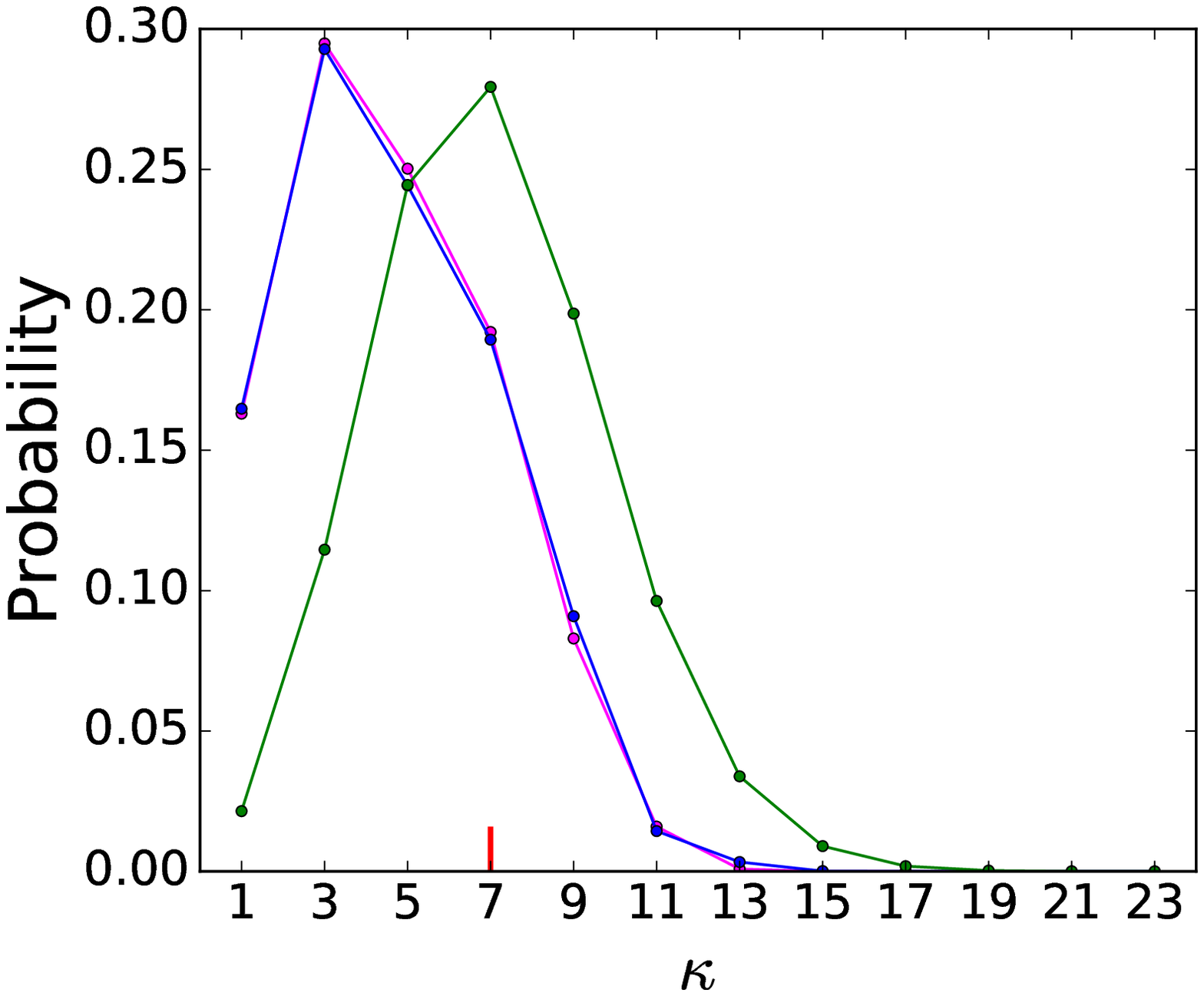} \\
(a) & (b)
\end{tabular}
\caption{\label{fig:DeConv_DataDim}(a) True sine-cosine series function with coeficients
$\beta_0 = 0.9,  \beta_1=\alpha_1=-0.4,  \beta_2=\alpha_2=-0.3, \beta_3=\alpha_3=-0.2, \beta_i=\alpha_i=0; i\geq4$
(red), its convolution (black) and simulated data points.
(b) Prior (green) and posterior probability of each dimension using the approximate FM (blue), complying
with the bound in (\ref{eqn:main_bound}), and using the exact FM (magenta); the true dimension ($\kappa=4$) is marked with red.} 
\end{figure}

\begin{figure}
\begin{tabular}{c c c}
\includegraphics[scale=0.25]{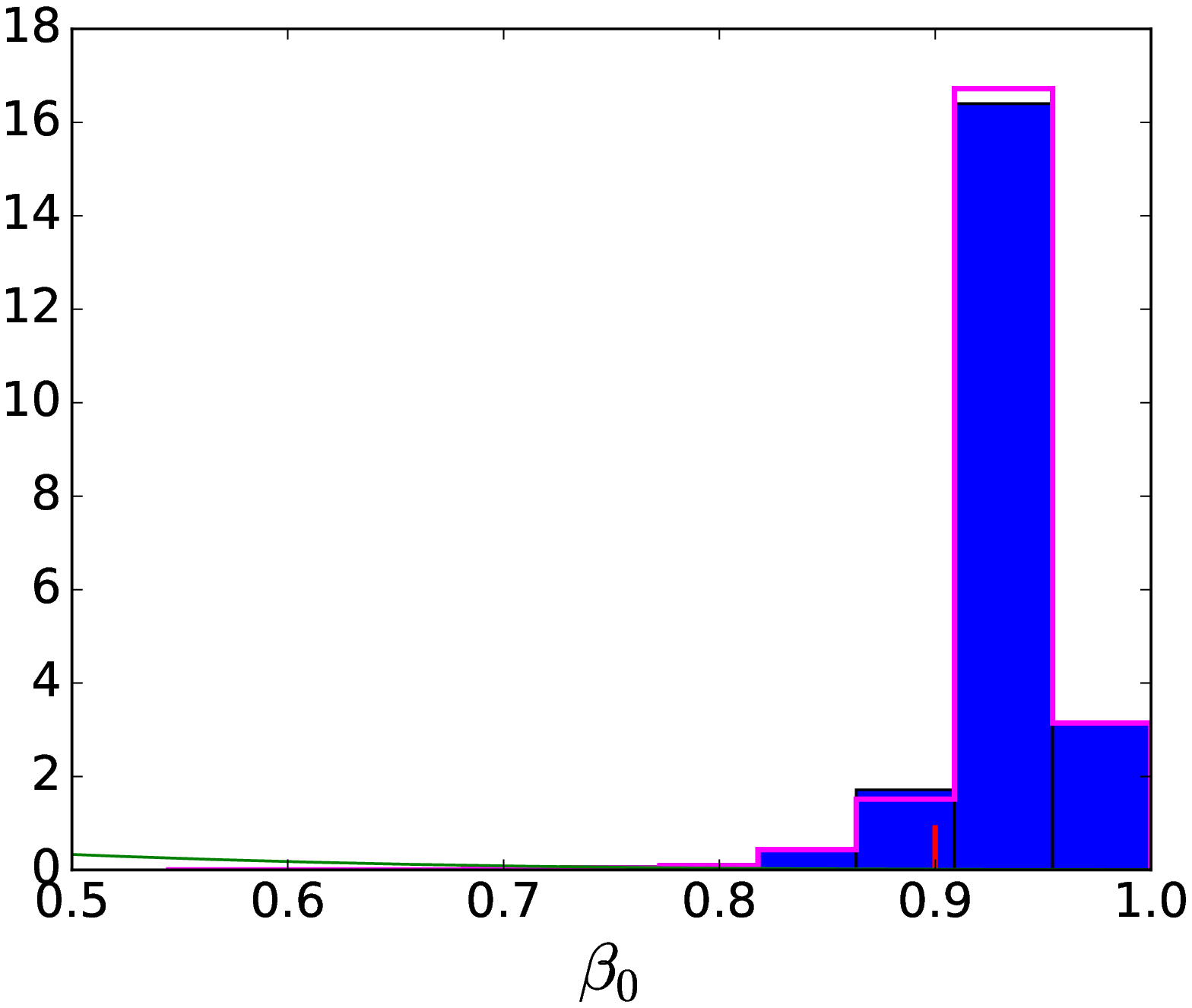} &
\includegraphics[scale=0.25]{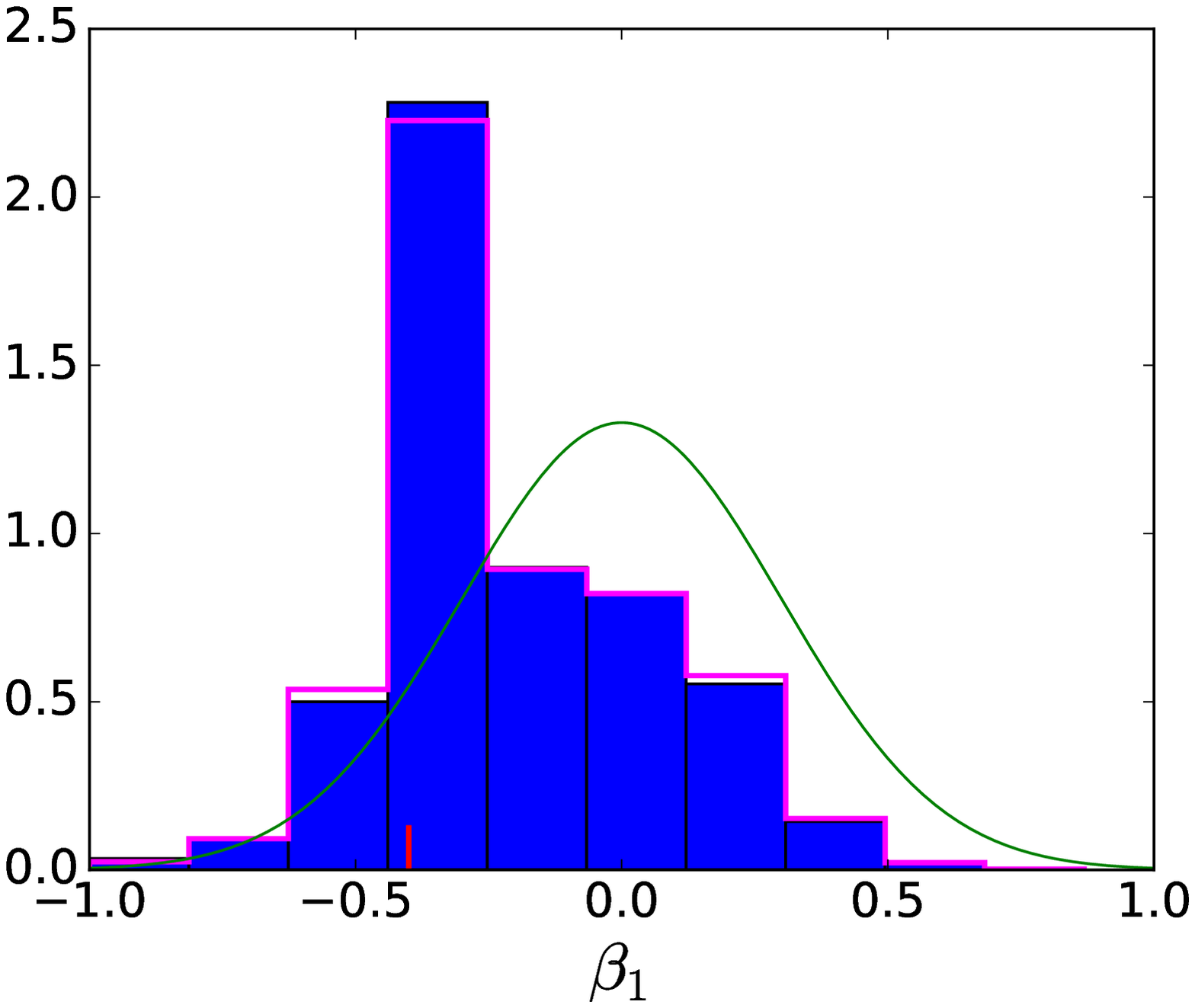} &
\includegraphics[scale=0.25]{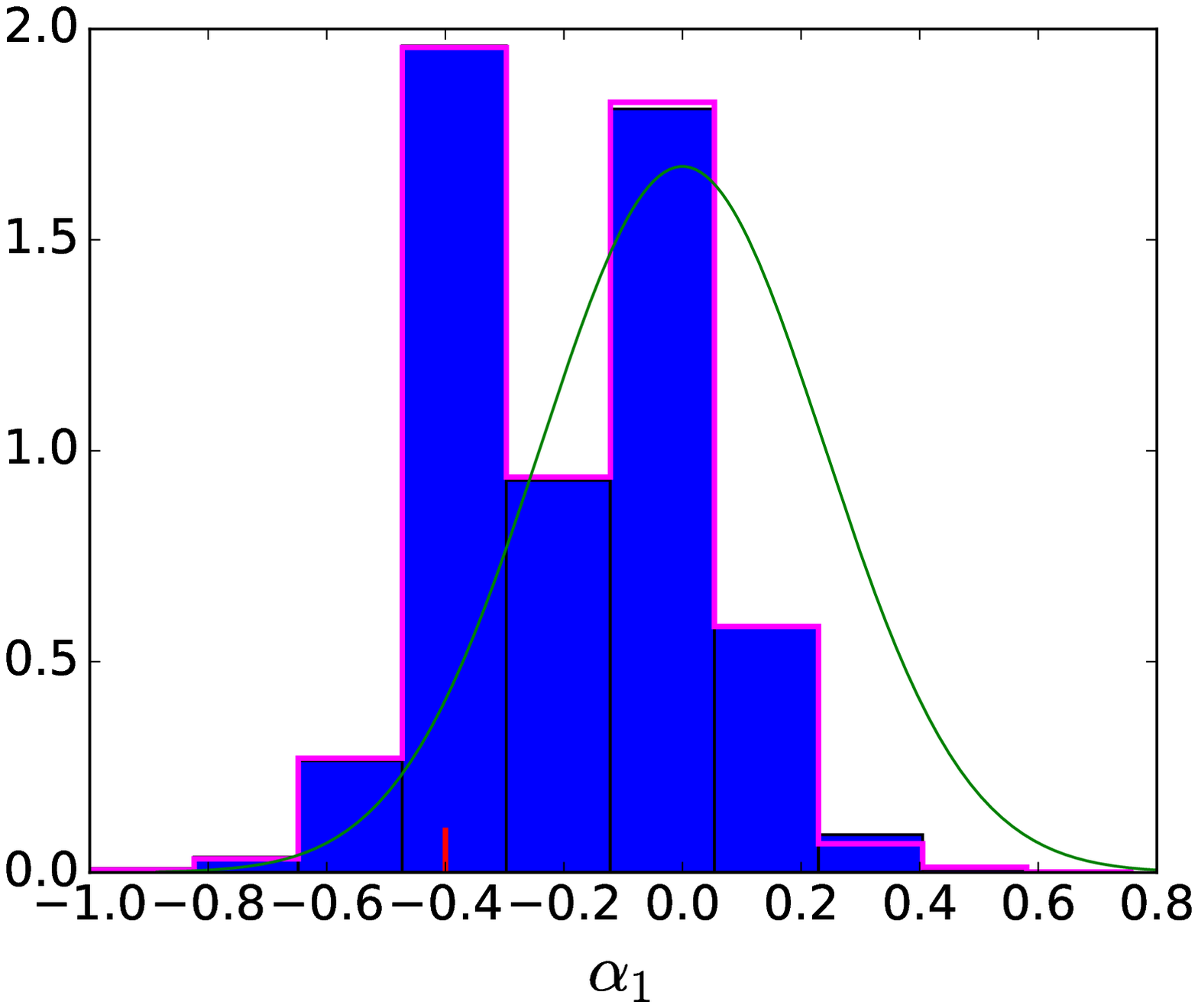} \\
\includegraphics[scale=0.25]{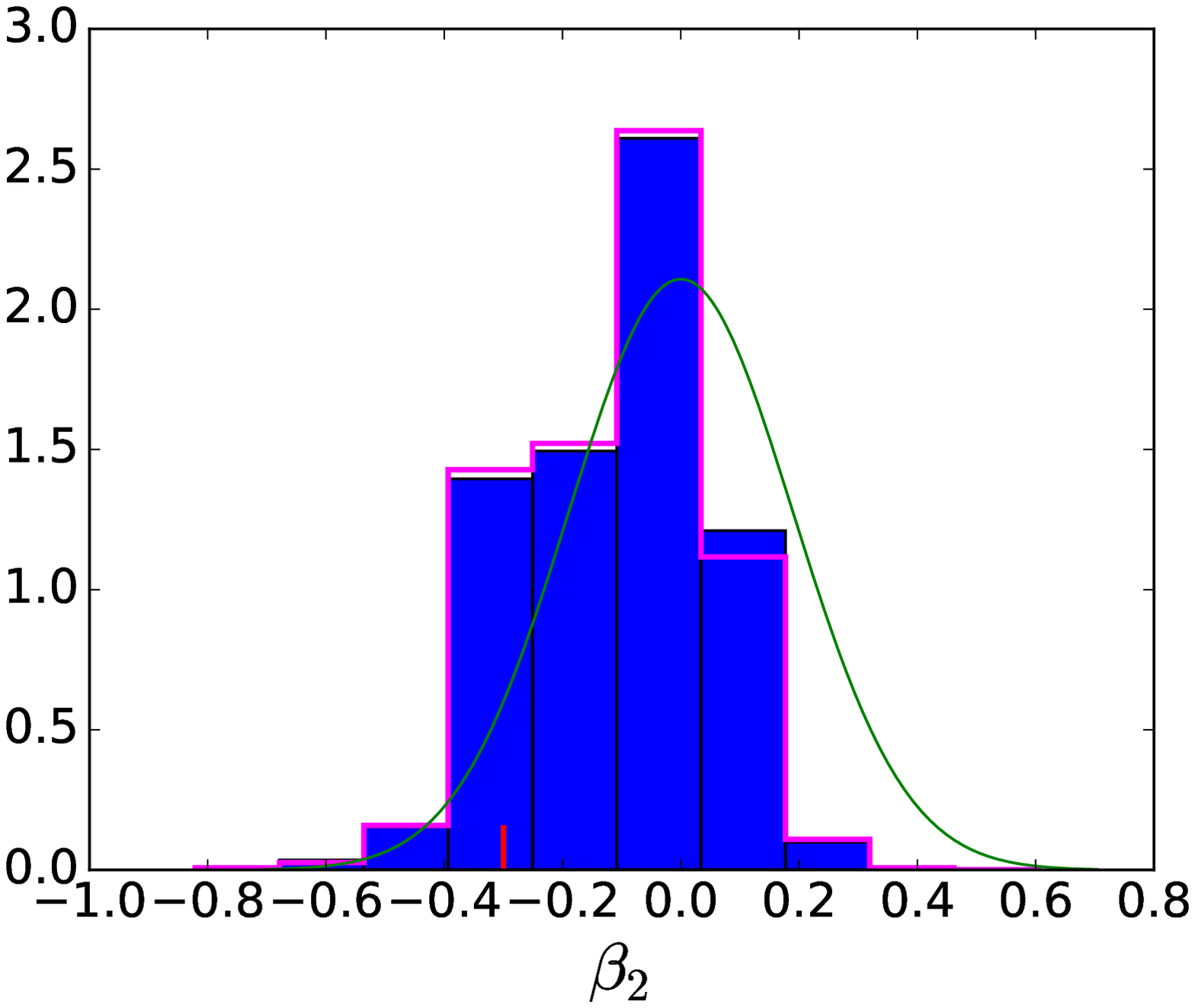} &
\includegraphics[scale=0.25]{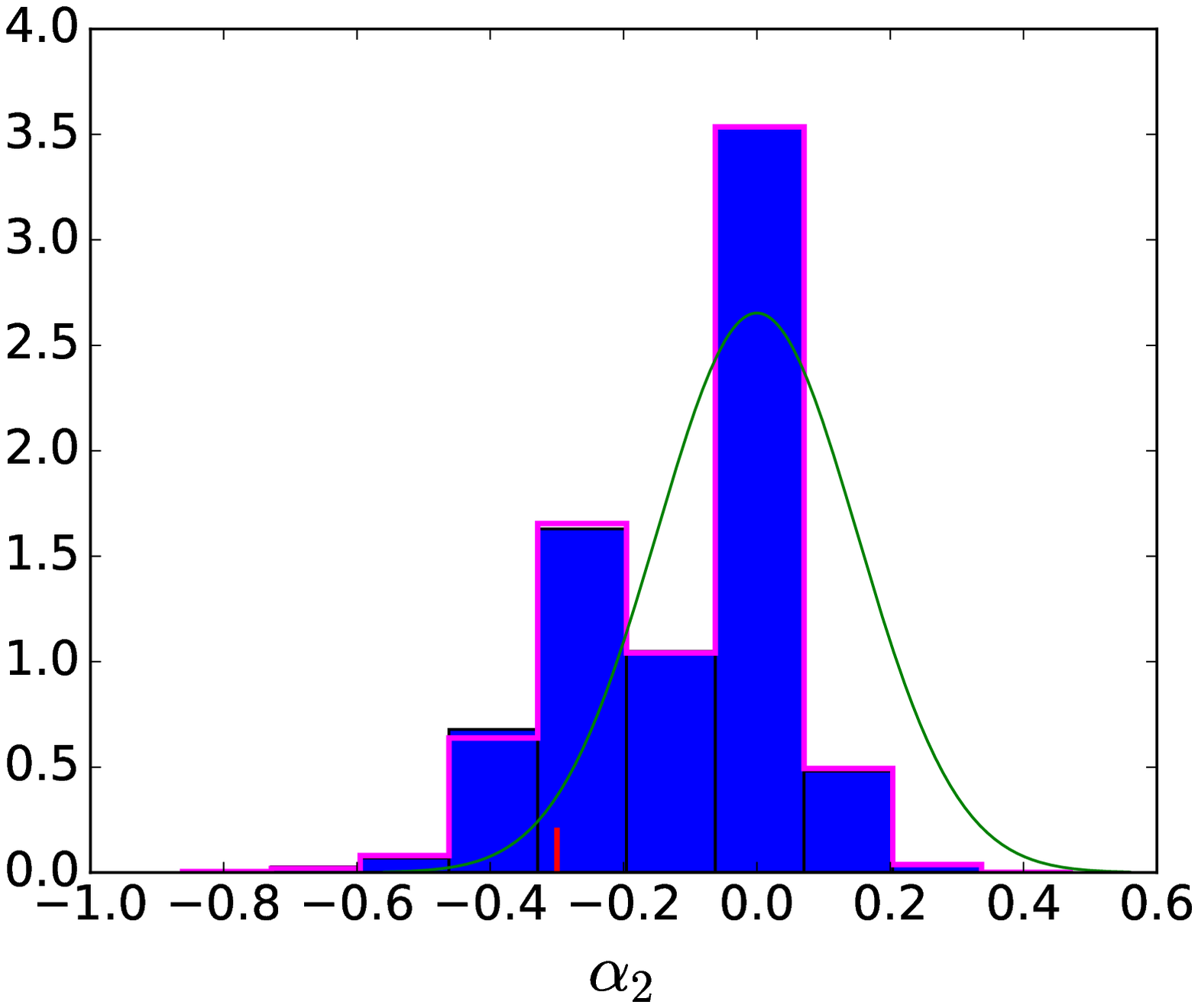} &
\includegraphics[scale=0.25]{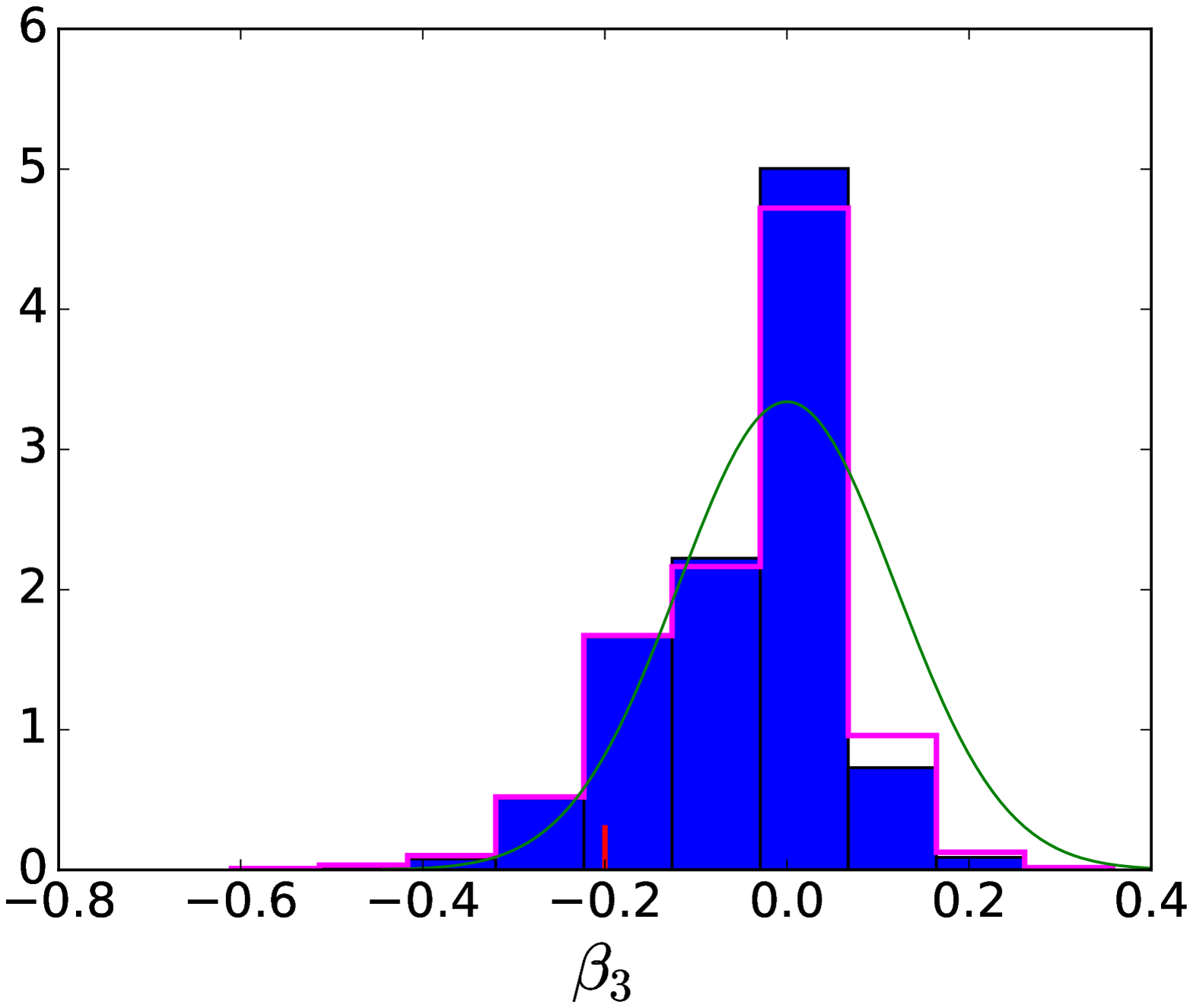} \\
\includegraphics[scale=0.25]{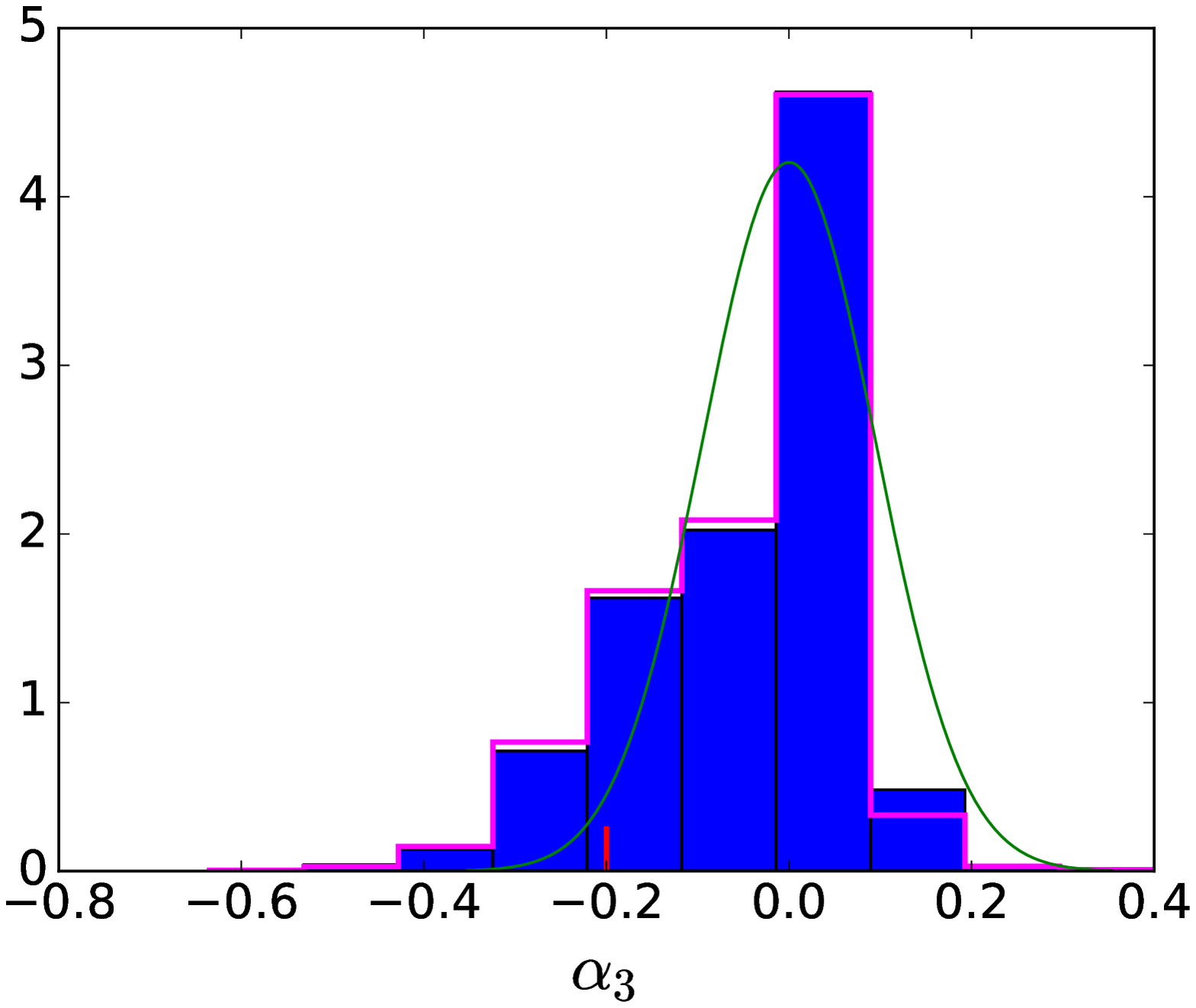} &
\includegraphics[scale=0.25]{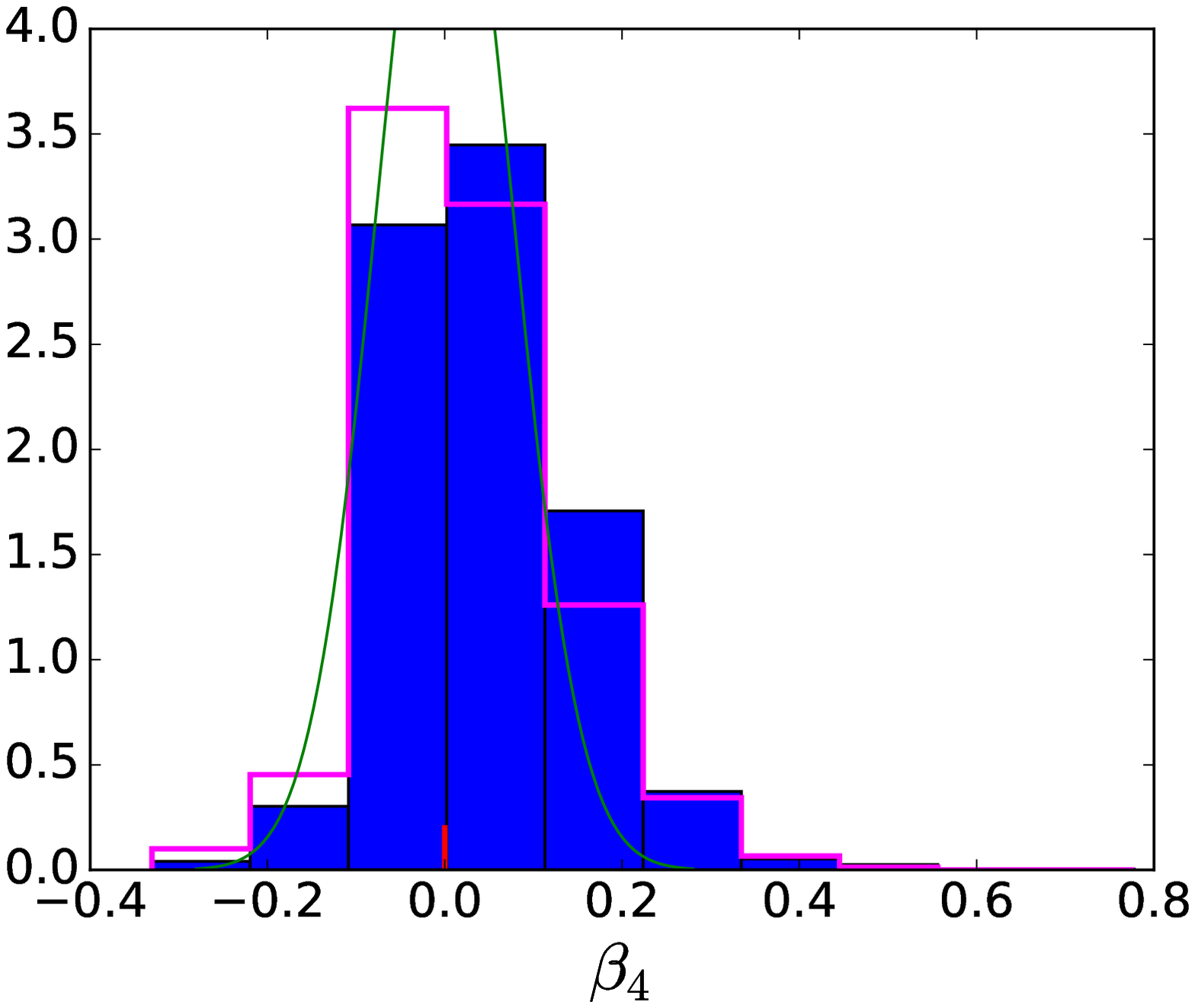} &
\includegraphics[scale=0.25]{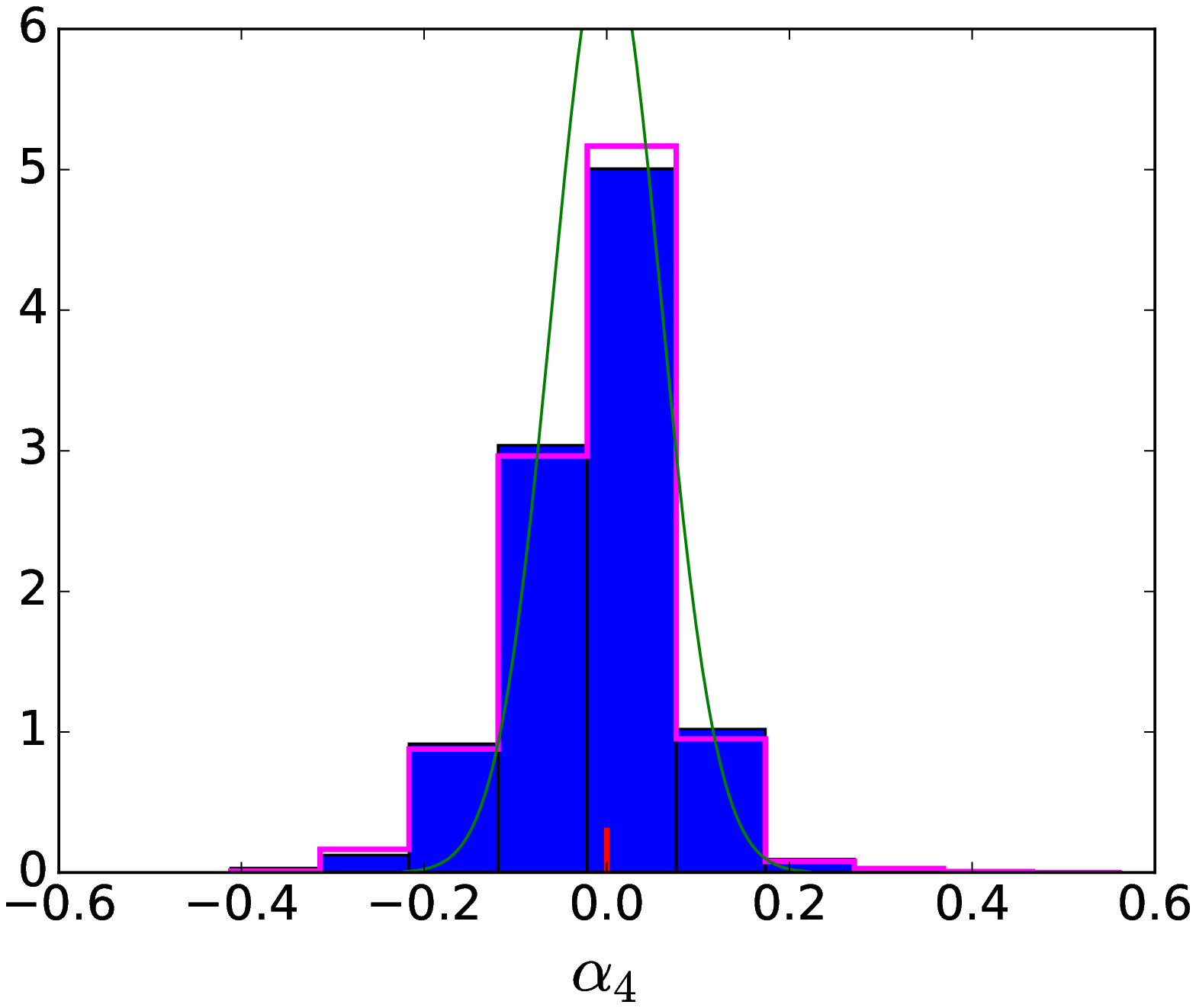} \\
\end{tabular}
\caption{\label{fig:DeConv_Posts} Prior (green) and posterior marginals for
$\beta_0, \beta_i, \alpha_i; i=1,2,3,4$ for the approximate FM (blue),
complying with the bound in (\ref{eqn:main_bound}), and using the exact FM (magenta).
The true value of the parameter is marked with a red tick.} 
\end{figure}

Our approximate FM leads to basically error free posteriors, as seen in
figures~\ref{fig:DeConv_DataDim}(b) and~\ref{fig:DeConv_Posts}.  Any extra
precision put into the Simpson's rule integrator will lead to useless extra CPU time,
with respect to the resulting numeric posterior, for the sample size and noise level
at hand.  For more realistic applications, where $\F[\theta_k]$ is not available analytically,
error bounds on the integrator could be used.  Moreover, since the error bound is required
at observations points $t_i$s only, an irregular integration grid could be used by making
it finer around the $t_i$s; this could lead to further improvements in CPU time.

\subsection{A 1D heat equation inferring the thermal conductivity}\label{sec:1d_pois}

Let us consider the thermal conductivity problem for the stationary heat equation in 1D 
\begin{align}
-\frac{d}{dx}\left(a\left(x\right)\frac{du\left(x\right)}{dx}\right)= & f(x),\qquad x\in\left(0,1\right),\label{eqn:heat}
\end{align}
subject to Dirichlet boundary conditions $u\left(0\right)=u\left(1\right)=0$, with forcing term
$f\left(x\right)=\sin\left(\pi x\right)$ and thermal conductivity $a(x) >0$ that varies with the space parameter $x$. 

In this example, the FM is not available analytically and a numeric (FEM) FM is used.  We use
an error estimation in the FM to bound the EABF.  In this case, since the FEM used is numerically
demanding we keep the prior truncation fixed ($\kappa = k$). 

The numerical solution of (\ref{eqn:heat}) is computed using the Finite Element Method (FEM), which allows us to calculate a local error estimation in the $L_{2}$ norm
\citep[see][for more details]{Babuvska1978}, given by
\[
\left\Vert u_{h}-u\right\Vert _{L_2(I_i)}=\left(\int_{x_{i-1}}^{x_{i}}\left(u_{h}-u\right)^{2}dx\right)^{1/2}\leq\frac{h^{2}}{\pi^{2}a_{\text{min}}^{i}}\left\Vert r\right\Vert _{L_2(I_i)}, \quad i=1,\ldots,n,
\]
where $m$ is the number of elements, $u_h$ the numerical solution with step size $h$, $I_{i}=[x_{i-1},x_{i}]$, $a_{\text{min}}^{i}=\ensuremath{\underset{x\in I_{i}}{\min} a\left(x\right)}$ and $r(x)=f(x)+\frac{d}{dx}\left(a(x)\frac{du_{h}(x)}{dx}\right)$ is the residual. Then, the the error estimation $\hat{K}_0$ is computed by
\begin{equation}
\hat{K}_0=\max_{I_{i}}\frac{h^{2}}{\pi^{2}a_{\text{min}}^{i}}\left\Vert r\right\Vert _{L_2(I_i)}, \quad i=1,\ldots,n.\label{eq:estk0}
\end{equation}

The inference problem is the estimation of the function $a(x) = \exp(b(x))$ given observations of  $u_{j}=u(x_{j})$ at a fixed locations $x_{j}$, $j=1,\ldots,m$.  Certainly, the theoretical and the numeric FMs are continuos.

We simulate a synthetic data set with the true thermal conductivity is $a\left(x\right)=k_{0}-r\frac{k_{0}}{1+\exp\left(-xa+\frac{a}{s}\right)}$, and error model $ Y_{j}=u\left(x_{j}\right)+\sigma\varepsilon_{j}$, where $\varepsilon_{j}\sim N\left(0,1\right)$, with the following parameters $k_{0}=5$, $r=0.9$, $a=20$, $s=2$ and $\sigma=0.0005$ (to maintain a 0.01 signal-to-noise ratio).  
The data are plotted in figure~\ref{fig:compar}(b). We consider $m=30$ observations at locations $x_{j}$ regularly spaced between $0$ and $1$.

In order to define the parametric space, the function $b$ is represented as a third-order b-spline that passes through the set of points $\{b_{i}\}_{i=0}^{k}$, where $b_{i}=b(x_{i})$. Therefore, the parameter space is defined by $\theta=\{b_{i}\}_{i=0}^{k}$.  In this case, the number of parameters
is taken as fixed $k=20$.  Regarding the prior distribution for the parameters $\{b_{i}\}_{i=0}^{k}$,
we define their prior using Gaussian Markov
random field (GMRF) zero mean and sparse precision matrix (inverse-covariance), encoding
statistical assumptions regarding the value
of each element $b_{i}$ based on the values of its neighbors
\citep[see details in][]{Bardsley2013}.  We restrict the support of $0 \leq b(x) \leq B$, that is $b_i \in [0,B]$, where $B=\log(10)$.
Then the parameter space is compact and there exist a global bound for (\ref{eq:estk0}), complying with
(\ref{eqn:global_error1}).

With the standard error and sample size used, calculating the error bound for the Forward Map (FM)
as stated in~(\ref{eqn:main_bound}), we
require $\hat{K}_0 < 2.1\times10^{-6}$. To sample from the posterior distribution, we also use the t-walk \citep{Christen2010}.

Regarding the numerical solver, we begin with a relatively large step size $h =0.02$
(considering $n=50$ elements in the FEM) and start the MCMC.  At each iteration the FM is first
computed along with its error estimation $\hat{K}_0$. If the solution $u_h$ do not satisfy the estimated global bound, ie. $\hat{K}_{0} > 2.1\times10^{-6}$, we increase the number of elements by $50$ ($h=1/(m+50)$), until the bound is met.  For $h = 0.0066$, $n=150$ elements in the FEM, the bound is achieved
for all iterations.  
For comparisons, a smaller grid is considered with $h = 0.002$, $n=500$ elements.
The results are shown in figure~\ref{fig:compar}.  We took 50,000 iterations of the twalk, the
MCMC mixes quite well.   With $n=150$ the sampling took 3 min and with $n=500$, 16 min;
in a standard 2.6Ghz processor computer.  As seen in  figure~\ref{fig:compar}
the conductivity is recovered and
taking $n=500$ elements in the FEM results in basically the same posterior as
for only $n=150$, which already comply with the EABF bound, only resulting in unnecessary
CPU effort. 

\begin{figure}
\begin{centering}
\begin{tabular}{c c}
\includegraphics[scale=0.30]{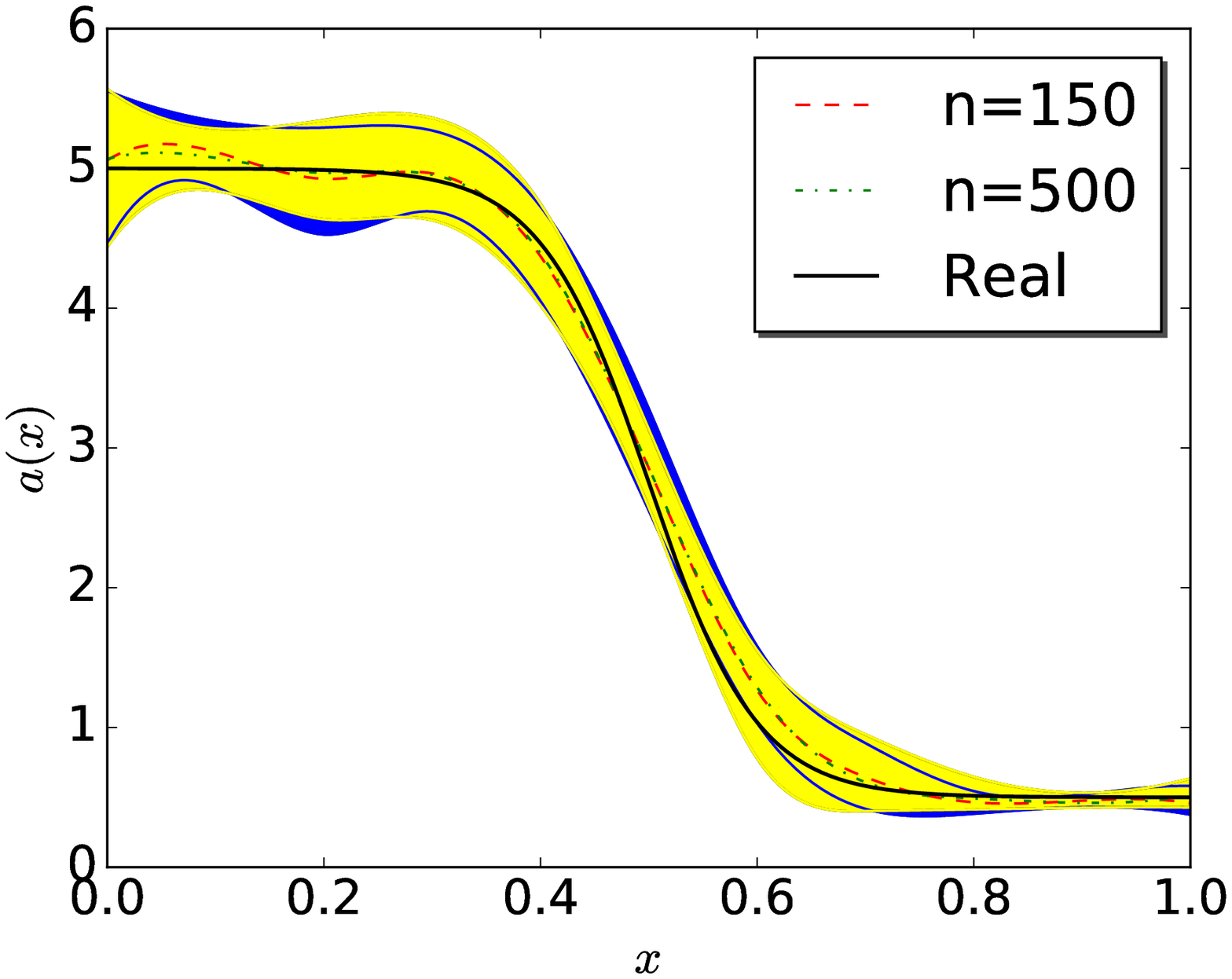} &
\includegraphics[scale=0.30]{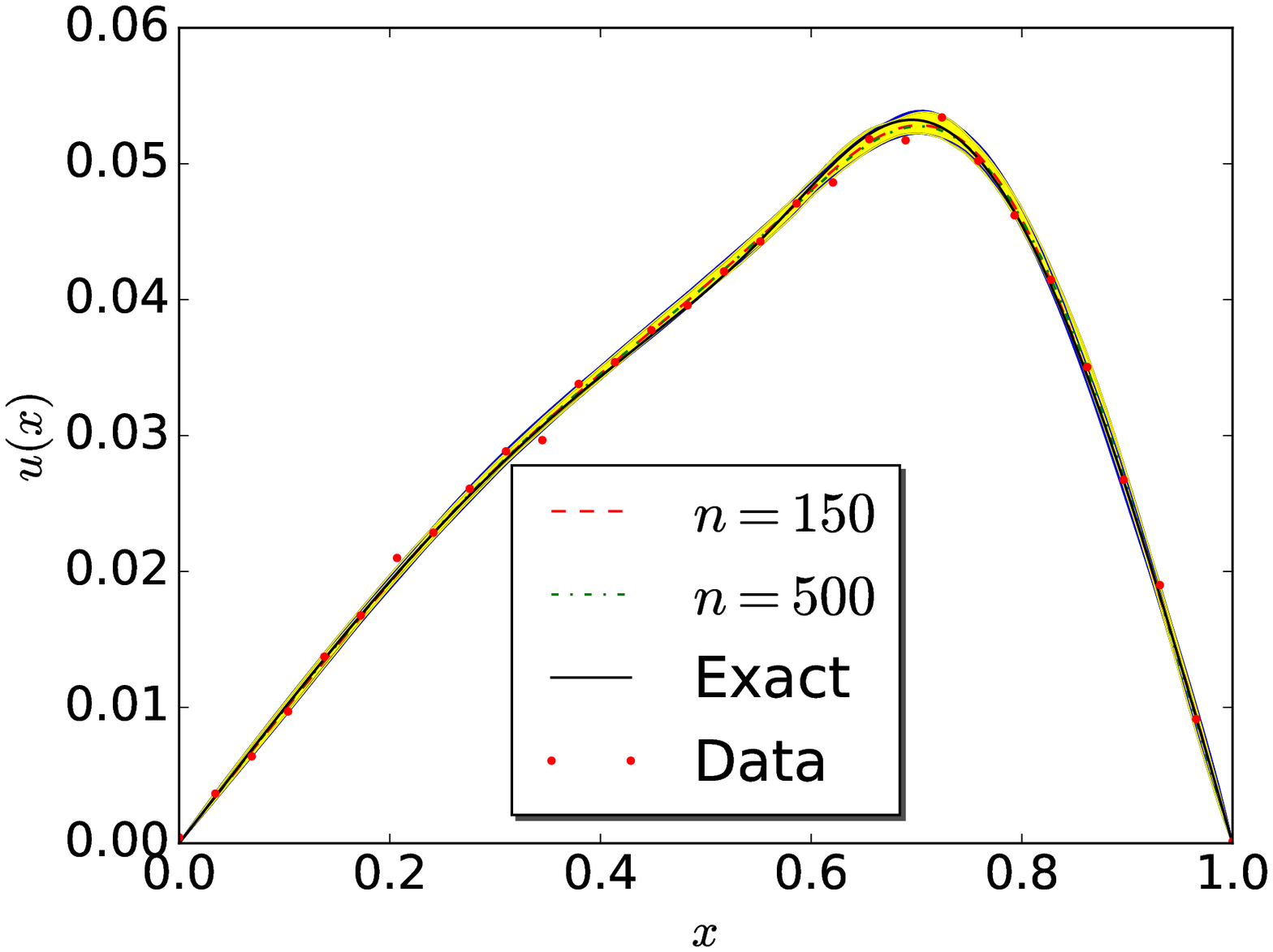} \\
(a) & (b) \\
\end{tabular}
\caption{\label{fig:compar} (a) The true conductivity $a(x)$ (black), the posterior mean with $n=150$ elements (red) and $n=500$ elements (green) in the FEM. (b) The exact solve $u(x)$ (black), the posterior mean with $n=150$ elements (red) and $n=500$ elements (green). Shaded areas represent the uncertainty in the model fit, as draws from the posterior distribution, using $150$ elements (blue) and $500$ elements (yellow). Note that, if we use a smaller step size than that required by the bound
in~(\ref{eqn:main_bound}), results are basically same simply adding CPU time.}
\par\end{centering}
\end{figure}

\subsection{A 2D heat equation inferring the initial condition}\label{sec:2d_pois}

We present a 2D heat equation problem to determine the initial
conditions from observations of transient temperature measurements taken within the
domain at a time $t=t_{1}$. The heat transfer PDE is given by
\begin{eqnarray}
\frac{\partial u}{\partial t} & = & \alpha\Delta u,\text{ in}\quad D=\left(0,1\right)\times\left(0,1\right),\label{eq:2Dheat}\\
u(x,y,t) & = & 0\quad\text{on\ensuremath{\quad\partial D}}\nonumber \\
u(x,y,0) & = & f\left(x,y\right).
\end{eqnarray}
Taking the forcing term $f\left(x,y\right)=b\sin\left(\pi x\right)\sin\left(\pi y\right)+c\sin\left(2\pi x\right)\sin\left(\pi y\right)$
as initial condition, the PDE has an analytical solution 
\[
u(x,y,t)=b\exp\left(-2\alpha\pi^{2}t\right)\sin\left(\pi x\right)\sin\left(\pi y\right)+c\exp\left(-5\alpha\pi^{2}t\right)\sin\left(2\pi x\right)\sin\left(\pi y\right) .
\]

In this example, we consider a more complex 2D PDE inverse problem, the FM is available analytically and a numeric FM is
also used; the numeric error is directly calculated. In this case, only two parameters are needed to be inferred.

A numerical solution of equation \ref{eq:2Dheat}(b) is also computed
using the Finite Element Method (FEM) within FEniCS \citep{AlnaesBlechta2015a},
which allows us to calculate the error in the numerical solver using the exact
solution.

\begin{figure}
\begin{center}
\begin{tabular}{c c c}
\includegraphics[scale=0.25]{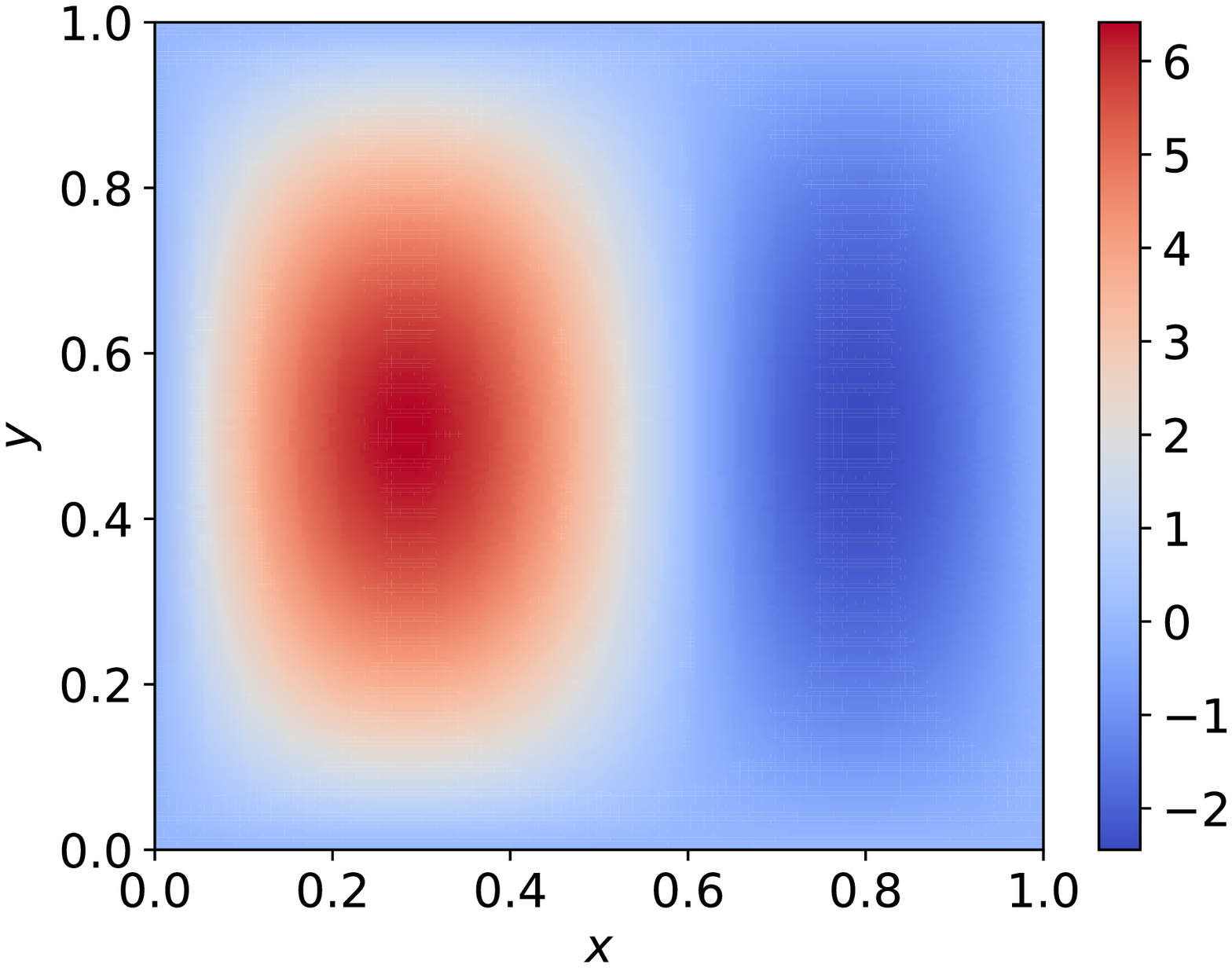} &
\includegraphics[scale=0.25]{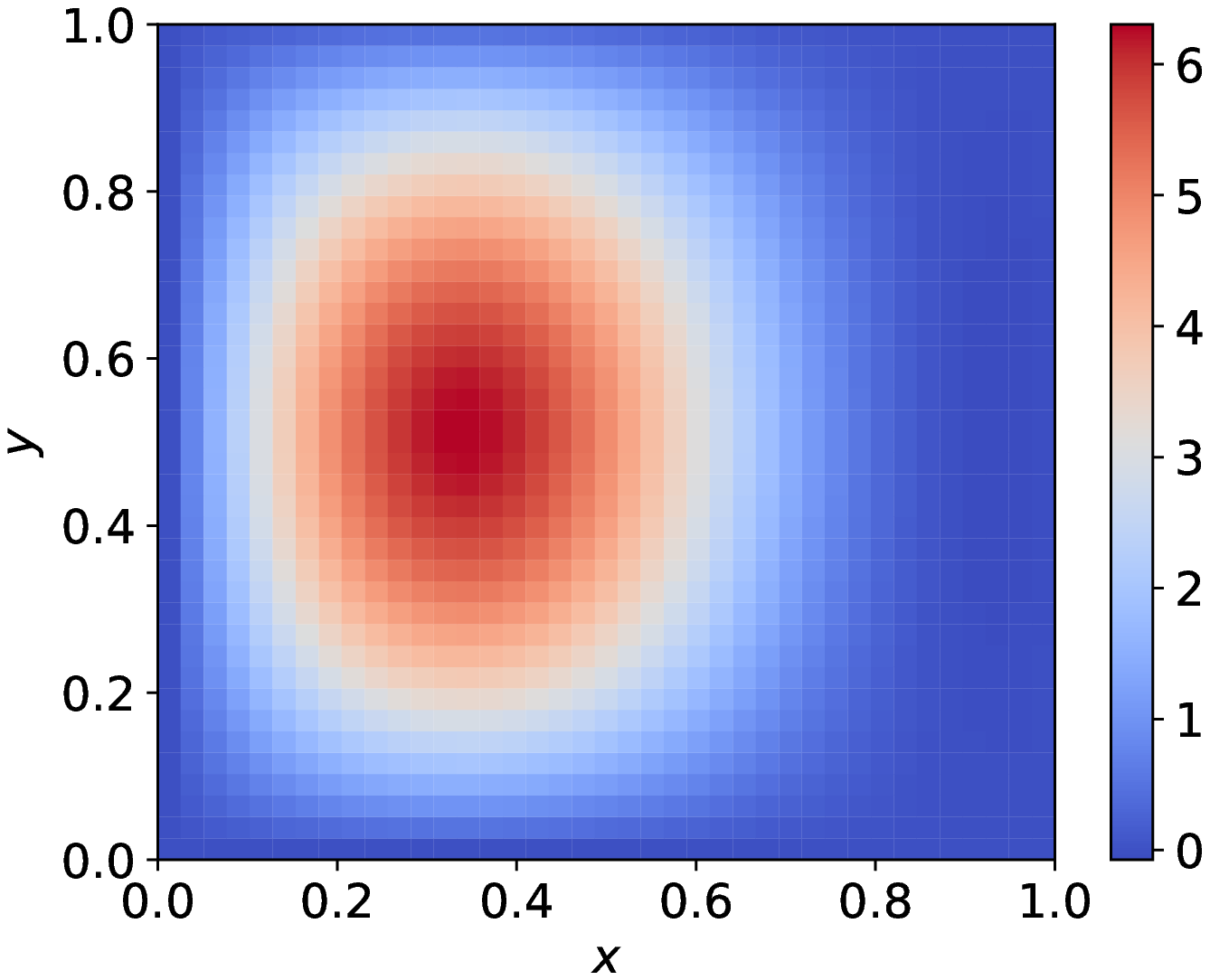} &
\includegraphics[scale=0.25]{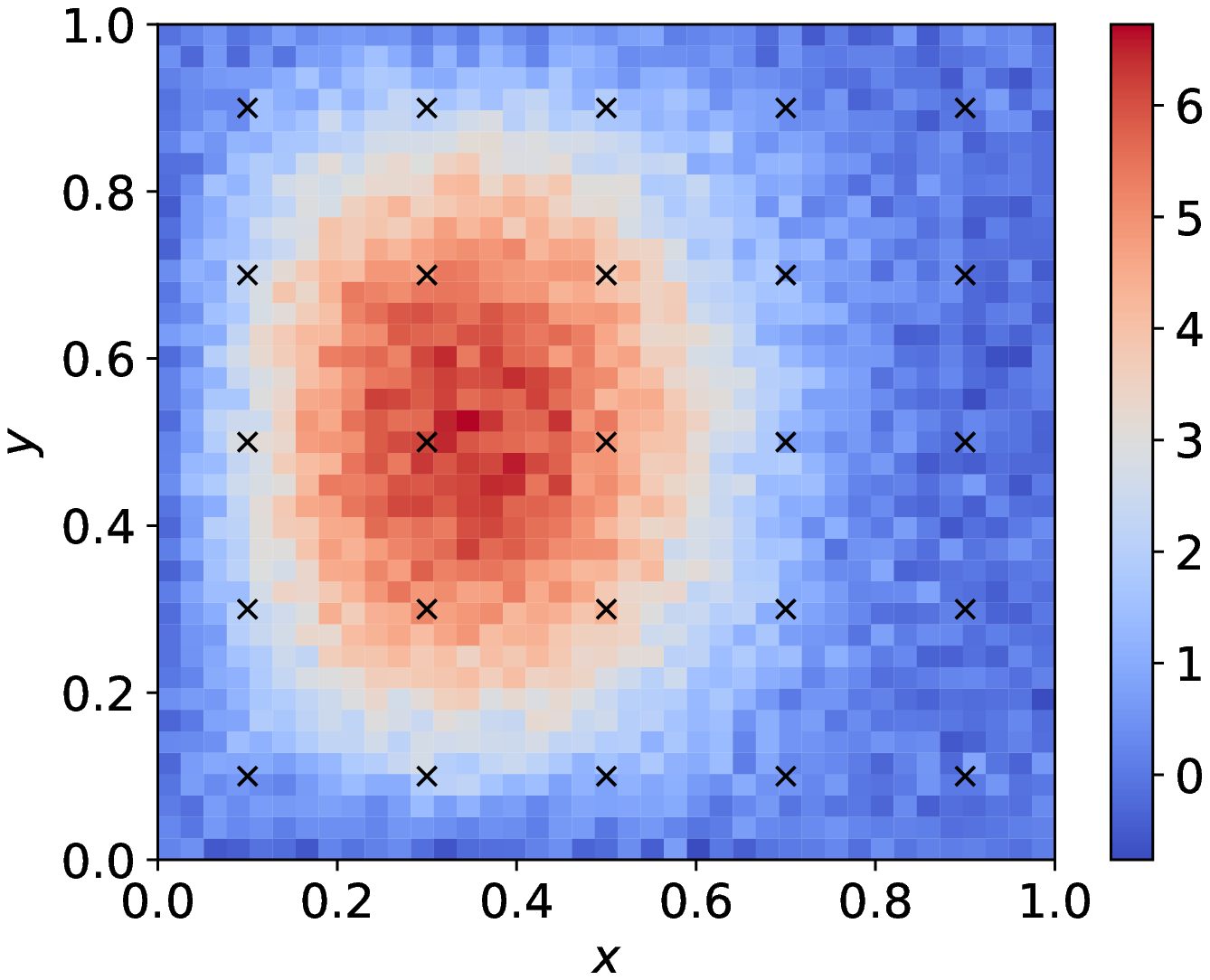} \\
(a) & (b) & (c) \\
\end{tabular}
\end{center}
\caption{Heat equation in 2D, (a) exact solution at $t=t_1$, (b) numerical solution using finite element
method with FEniCS with mesh $40\times40$ with $\Delta t=0.067$ and (c)
numerical solution with an additive noise gaussian with variance $\sigma=0.3$ and data
point locations.
\label{fig:sol_EDP}}
\end{figure}

The inferential problem is to estimate $\theta=\left(b,c\right)$ given
measurements of $u$ at time $t_1=0.3$.  A priori we took independent truncated Gamma distributions
for $b$ and $c$ with parameters $(2,0.7)$ and $(2,0.4)$ respectively, both restricted to $[0,8]$.
Certainly, the theoretical and the numeric FMs are continuos, and since the support is compact
we may conclude that the error bound in (\ref{eqn:global_error1}) exists for all $\theta$.

We simulate a synthetic
data set with the error model 
\[
Y_{i} = u( x_i, y_i,  t_1 ) + \sigma\varepsilon_{i},
\]
where $\varepsilon_{i}\sim N\left(0,1\right)\quad i=1,\ldots,n$,
$\sigma=0.3$ (using a the signal to noise ratio of $5\%$), with $b=3$ and $c=5$.
The data are plotted in Figure \ref{fig:sol_EDP}(b).
We consider $n=25$ observations, $( x_i, y_i ),\quad i,\ldots,n$
regularly spaced on $D$.  Since we have an analytic
solution, if we run the PDE solver we may calculate
the maximum absolute error, $K_{0}$, exactly.
The error bound for the FM as stated in~(\ref{eqn:main_bound})
is $\simeq0.0015$. To sample from
the posterior distribution we use the t-walk \citep{Christen2010}.

Regarding the numerical solver
we start with a large step size of $\Delta x=\Delta y=0.1$ and $\Delta t=0.268$,
and calculate $K_{0}$. If the solution does not comply with the bound,
that is $K_{0}>0.0015$, a new solution is attempted by reducing the
step-size in $\Delta x,\Delta y$ and $\Delta t$ by half, until the
global absolute errors is within the bound, $K_{0}\leqslant0.0015$.
The resulting mesh is $\Delta x=\Delta y=0.025$ and $\Delta t=0.067$

We compare the above FEM numerical FM with the exact FM, with
250,000 iterations of our MCMC.
The result are shown in Figure~\ref{fig:Comp_Th_Num} and in Table~\ref{tab:Comp_Th_Num}.
The differences observed in both results may be attributed to the
Monte Carlo sampling.  

\begin{figure}
\begin{tabular}{c c}
\includegraphics[scale=0.4]{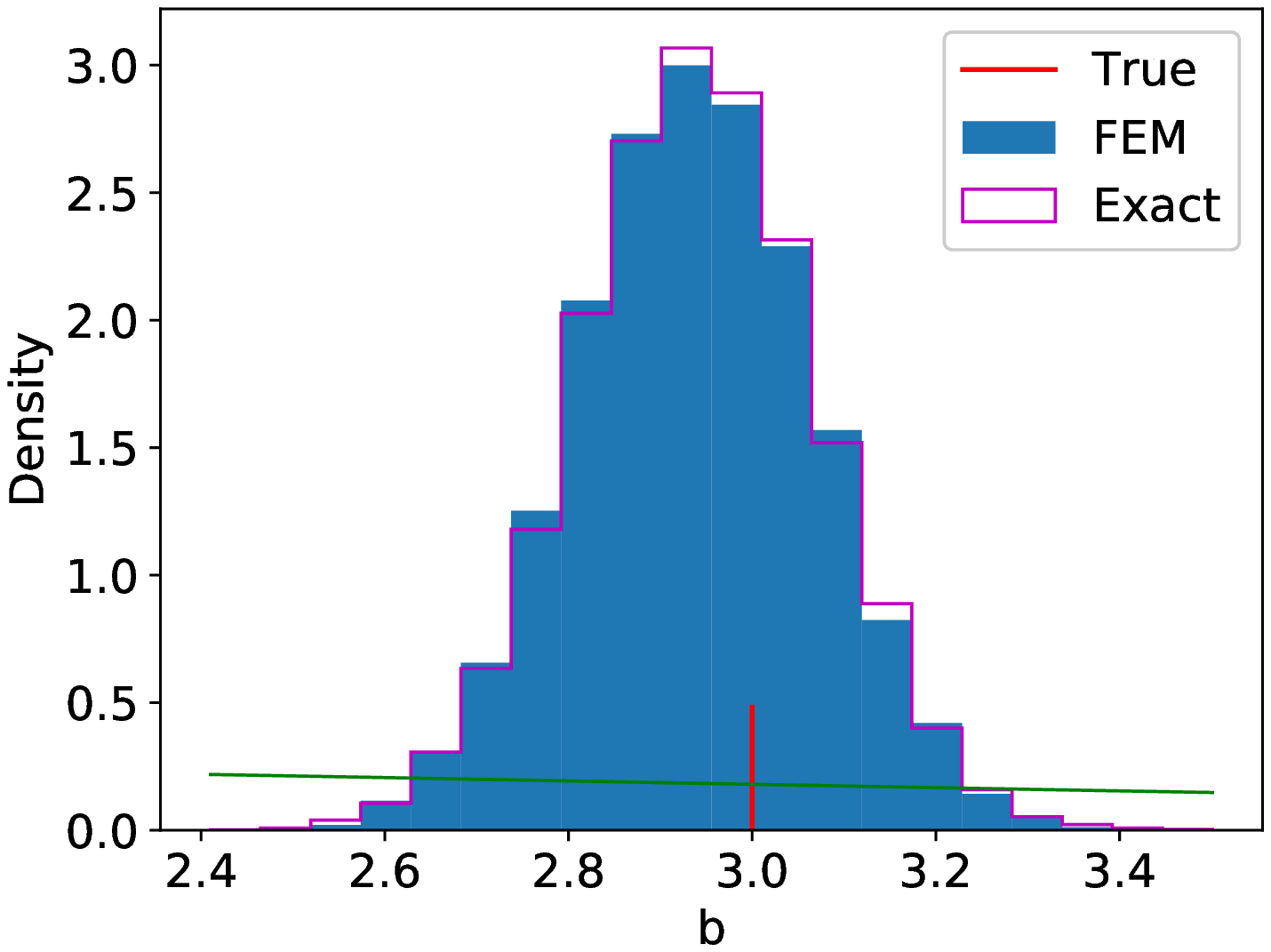} &
\includegraphics[scale=0.4]{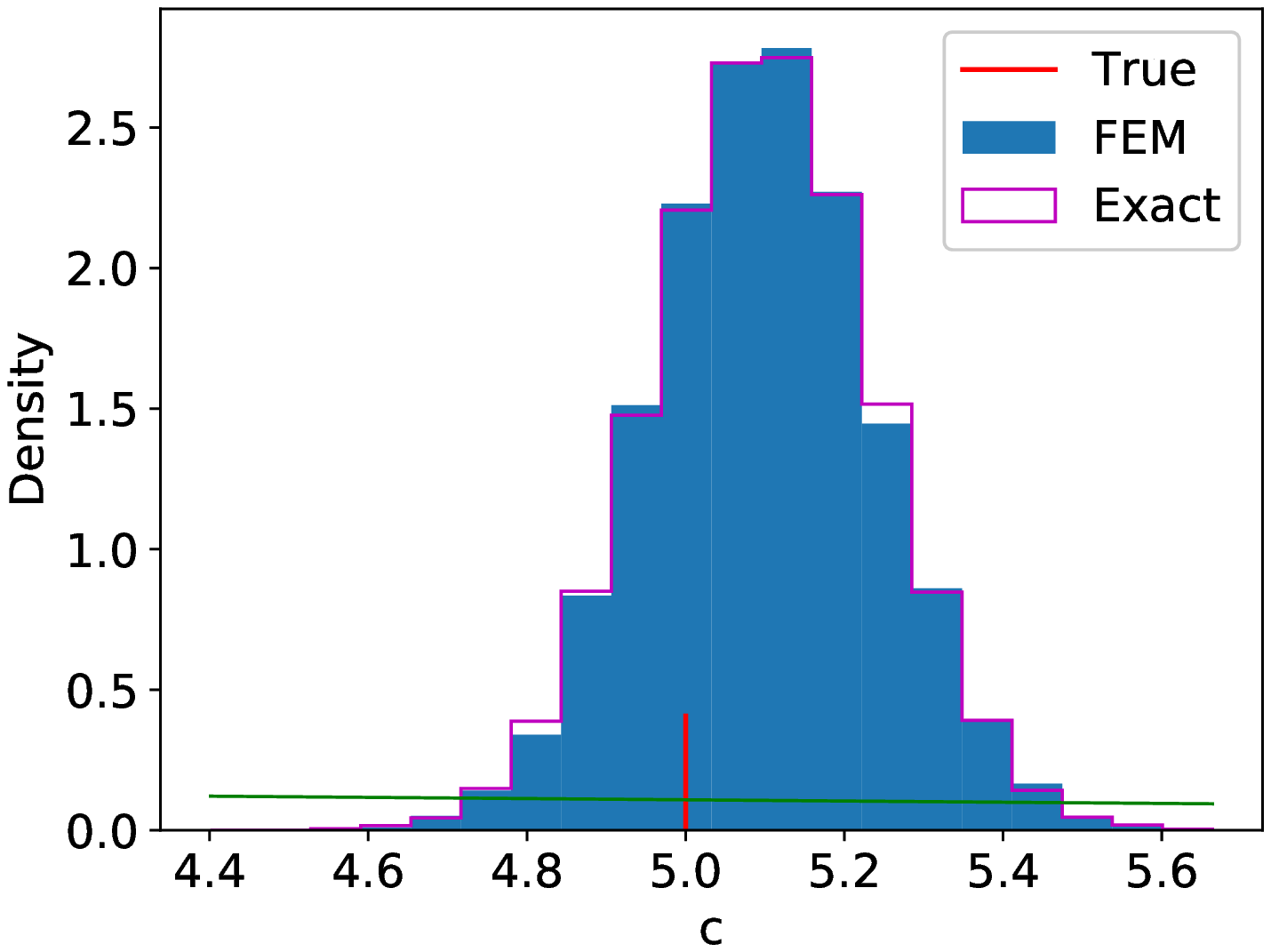} \\
\end{tabular}
\caption{Comparison between numerical (blue) a theoretical (magenta) posteriors for both parameters in the initial conditions of the 2D heat equation.\label{fig:Comp_Th_Num} }
\end{figure}

\begin{table}
\begin{center}
\begin{tabular}{ l  l  l }
\hspace{1cm}  & $b$ \hspace{2cm} &  $c$ \hspace{2cm} \\
\hline \hline
True          & 3.0       & 5.0       \\
PM-Exact \hspace{1cm}  & 2.9396 & 5.0966 \\
PM-FEM   \hspace{1cm}  & 2.9377 & 5.0969 \\
\hline \hline
\end{tabular}
\end{center}
\caption{Comparison of the Posterior
Mean (PM) of parameters $b$ and $c$ using the exact FM and
the FEM approximate FM.\label{tab:Comp_Th_Num}}
\end{table}

\section{Discussion}\label{sec:diss}

The generalization of the results of \cite{Capistran2016} to a priori statements,
Banach parameter spaces and a truncation in the prior makes the error control strategy, ie. using BFs, of the latter far more
feasible, general and applicable.
 
In passing we needed to define the posterior distribution in this general setting and prove its existence,
as presented in section \ref{sec:weak}.   However, this we did
using standard results in probability and modern Bayesian theory.  Regarding the finite dimensional numeric posterior,
weak convergence is then not difficult to prove and also TV rates of convergence are proved to be maintained, as seen in
theorems \ref{teo:tvn_rate}  and \ref{teo:tvk_rate}, this relaying on lemmas \ref{lemm:conv_post1}
and  \ref{lemm:conv_post2}.

We have not discussed the scenario when error parameters $\sigma$ are not known.  In this case
we may consider that a priori $\theta$ and $\sigma$ are independent and
equivalent results should follow; this was discussed in a previous version of this manuscript but not here \citep{Christen2017}.  We only need to prove that the likelihood including $\sigma$ follows assumption \ref{assp:1}, in particular that it is bounded $\lambda$-a.s.

We have not proved that stylized posterior estimates like the mean or variance exists for $\Q_y[\pi]$.
Elsewhere, these are proven to exists with additional requirements and for Gaussian priors, or with exponential tails, using Fernique's theorem \citep{Stuart2010, Hosseini2017}.
In our case, an additional sufficient requirement is mentioned in (\ref{eqn:uni_int}), which only involves the finite dimensional measures $\Q_y^n[\pi_k]$,
which can be examined in a case-by-case basis.  Note, however,
that as far as Bayesian inference is concerned, we need not to guarantee the existence of the posterior expected mean, variance etc.
adding regularity conditions on the observational model and/or on the prior.
If, for example, a posterior distribution has no variance, that is a very relevant and important information regarding the
statistical inference problem at hand.  Nonetheless,
all posterior probabilities and posterior expected utilities are proven to be
consistent and well defined, given weak convergence and TV convergence rates.

We presented 4 workout examples of increasing difficulty. In all cases, the numerical error
in the posterior was controlled successfully leading to negible increase in precision
if a more precise FM is considered.  This in turn may result in CPU time save, as cheaper/rougher
solvers are used.  Note that decreasing solver precision can only be done within limits,
that is within the stable regime of the solver used.  Moreover, in real case applications,
increasing the mesh size or any mesh refinements come a great coding effort, for
example in a large scale 3D geothermal inversion \citep{CUIetAl2011}.
Our approach only makes sense in the case where mesh refinements and
reliable after the fact error estimates are readily available.

\section{Acknowledgments}

We thank Tan Bui-Thanh (UT Austin) form prompting us to work on this generalization and for several comments on a previous draft of the paper.  Also to Peter M\"uller (UT Austin), Jos\'e Luis Perez Garmidia,
and Fernanda M\'endez (CIMAT) for invaluable comments during the many previous drafts of the paper.  This research is partially founded by CONACYT CB-2016-01-284451,
RDECOMM and ONRG grants.

\bibliographystyle{chicago}
\bibliography{PostErrControl_biblio}

\appendix

\section{Auxiliary lemmas}

\begin{lemma}\label{lemm:conv_post1}
Let $b_n, b: \Theta \rightarrow \R^+$ be bounded, $\pi$-integrable functions and let
$z_n = \int b_n(\theta) \pi(d\theta)$ and $z = \int b(\theta) \pi(d\theta)$ and assume that
$| b_n(\theta) - b(\theta) | < k n^{-p}$ for all $\theta \in \Theta, n > N$; $k>0, p>1$ fixed.  Then
$z_n \rightarrow z$ and $\frac{b_n(\theta)}{z_n} \rightarrow \frac{b(\theta)}{z}$ with convergence rates
$$
|z_n - z| < k n^{-p} ~~\text{and}~~
\left| \frac{b_n(\theta)}{z_n} - \frac{b(\theta)}{z} \right| < \frac{b(\theta)}{z} \frac{k}{z}  n^{-p} + \frac{k}{z} n^{-p},
$$
for all $\theta \in \Theta$ and big enough $n$. 
\end{lemma}

\begin{proof}
Since $b_n(\theta) \leq l(\theta)=M$ then by dominated convergence $z_n \rightarrow z$
(since $M = \int l(\theta) \pi$).  Since $z_n, z >0$, $b_n \rightarrow b$ already implies
$b_n/z_n \rightarrow b/z$.

Now, for the rate of convergence we have
$$
|z_n - z| = \left| \int b_n(\theta) \pi(d\theta) - \int b(\theta) \pi(d\theta) \right| <
\int \left| b_n(\theta) - b(\theta)  \right| \pi(d\theta) < k n^{-p} ,
$$
since $\int \pi(d\theta) = 1$, and therefore $\frac{| z_n - z|}{z} < \frac{k}{z} n^{-p}$. 

Note that the first order Taylor series with residual of $x^{-1}$ around $x_0$ is
$x^{-1} = x_0^{-1} - x_0^{-2}(x - x_0) + x_1^{-3} (x - x_0)^2$, for $x_1$ between $x$ and $x_0$.
Assuming $x,x_0> 0$ then 
$$
\frac{|x^{-1} - x_0^{-1}|}{x_0^{-1}} \leq \frac{|x - x_0|}{x_0} +  \left( \frac{x_0}{x_1} \right)^{3} \left( \frac{x - x_0}{x_0} \right)^2 .
$$
Let $\frac{|x - x_0|}{x_0} < \epsilon$, then also $\frac{|x - x_0|}{x_0} < \epsilon$ and
$(1- \epsilon) < \frac{x_1}{x_0} < ( 1 + \epsilon)$.  Since $x^{-3}$ is decreasing then
$( 1 + \epsilon)^{-3} < \left( \frac{x_0}{x_1} \right)^{3} < (1- \epsilon)^{-3}$.  Therefore
$\left( \frac{x_0}{x_1} \right)^{3} \left( \frac{x - x_0}{x_0} \right)^2 <  (1- \epsilon)^{-3} \epsilon^2$.
If the relative error $\epsilon$ (of estimating $x_0$ with $x$) is below 20\%, $(1- \epsilon)^{-3} \epsilon^2$
is already one order of magnitud smaller than $\epsilon$.
Then ignoring this last term
$$
\frac{|x^{-1} - x_0^{-1}|}{x_0^{-1}} \lesssim \frac{|x - x_0|}{x_0} < \epsilon .
$$

Assume $n$ is big enough such that the relative error $\frac{| z_n - z|}{z} < \frac{k}{z} n^{-p}$ is
small enough and we have $\frac{| z_n^{-1} - z^{-1} |}{z^{-1}}  \lesssim \frac{| z_n - z|}{z} < \frac{k}{z} n^{-p}$.
Therefore
\begin{equation}\label{eqn:zbound}
z^{-1} - \frac{k}{z^2} n^{-p} < z_n^{-1} < z^{-1} + \frac{k}{z^2} n^{-p}.
\end{equation}
Since $b(\theta) - k n^{-p} < b_n(\theta) < b(\theta) + k n^{-p}$, and given that
we may assume $0 < z^{-1}(1 - \frac{k}{z} n^{-p})$, then multiplying (\ref{eqn:zbound}) with the
fomer term we have
$$
(b(\theta) - k n^{-p}) (z^{-1} - z^{-2} k n^{-p}) < \frac{b_n(\theta)}{z_n} < (z^{-1} + z^{-2} k n^{-p}) (b(\theta) + k n^{-p}) ,
$$
\begin{eqnarray*}
& & b(\theta)z^{-1} - b(\theta) z^{-2} k n^{-p} - k n^{-p}z^{-1} + z^{-2} k^2 n^{-2p} < \frac{b_n(\theta)}{z_n} < \\
& & b(\theta)z^{-1} + b(\theta) z^{-2} k n^{-p} + k n^{-p}z^{-1} + z^{-2} k^2 n^{-2p} .
\end{eqnarray*}
Ignoring the two terms of $2p$ order, we obtain the result.
\end{proof}

\begin{lemma}\label{lemm:conv_post2}
With the setting of lemma \ref{lemm:conv_post1}, let $h(\theta)$ measurable and
$\hat{h}_n = \int h(\theta) \frac{b_n(\theta)}{z_n} \pi(d\theta)$ and
$\hat{h} = \int h(\theta) \frac{b(\theta)}{z} \pi(d\theta)$ exists.  Then
$$
| \hat{h}_n - \hat{h} | < E_1[| h |] \frac{k}{z} n^{-p} +   E_0[| h |] \frac{k}{ z} n^{-p}
$$
where $E_1[| h |] = \int | h(\theta) | \frac{b(\theta)}{z} \pi(d\theta)$ and
$E_0[| h |] = \int | h(\theta) | \pi(d\theta)$.  Moreover, for all $h$
non-negative and bounded, $\int h(\theta) \frac{b_n(\theta)}{z_n} \pi(d\theta)$ and
$\int h(\theta) \frac{b(\theta)}{z} \pi(d\theta)$ implicitly define the probability measures
$p_n$ and $p$, then
$$
||p_n - p||_{TV} < \frac{k}{z} n^{-p} .
$$
\end{lemma}

\begin{proof}
We have
$$
| \hat{h}_n - \hat{h} | \leq \int |h(\theta)|  \left| \frac{b_n(\theta)}{z_n} - \frac{b(\theta)}{z} \right| \pi(d\theta) 
$$
and using lemma \ref{lemm:conv_post1} we obtain the first result.  Moreover, if $|h| \leq 1$
then $\hat{h} \frac{k}{z} n^{-p} + \hat{h}_0 \frac{k}{z} n^{-p} <  2\frac{k}{z} n^{-p}$
and therefore
$$
\frac{1}{2} \max_{|h| \leq 1} \left| \int h(\theta) p_n(d\theta) -  \int h(\theta) p(d\theta) \right| <
\frac{1}{z} k n^{-p} 
$$
and we obtain the second result.
\end{proof}

\end{document}